\documentclass[11pt,reqno]{amsart}
\usepackage[T1]{fontenc}
\usepackage[utf8]{inputenc} 
\usepackage{lmodern}
\usepackage[english]{babel}
\usepackage{microtype}

\usepackage[%
    a4paper,
	left=2.5cm,       
	right=2.5cm,      
	top=3.5cm,        
	bottom=3.5cm,     
	heightrounded,    
	bindingoffset=0mm 
]{geometry}

\usepackage{graphicx} 
\usepackage{float} 

 \usepackage[parfill]{parskip} 
 
\usepackage{booktabs} 
\usepackage{array} 
\usepackage{paralist} 
\usepackage{verbatim} 
\usepackage{subfig} 
\usepackage{mathrsfs}
\usepackage{amssymb}
\usepackage{xcolor}
\usepackage{amsthm}
\usepackage{amsmath,amsfonts,amssymb,esint,hyperref}
\usepackage[noabbrev, capitalize]{cleveref}

\usepackage{autonum}

\usepackage{graphics,color}
\usepackage{enumerate, enumitem}
\usepackage{mathtools,centernot}
\usepackage{cases}
\usepackage{amsrefs}
\usepackage{bbm}
\usepackage{xfrac}

\usepackage{bookmark}

\newtheorem{theorem}{Theorem}[section]

\newtheorem{corollary}[theorem]{Corollary}
\newtheorem{proposition}[theorem]{Proposition}
\newtheorem{remark}[theorem]{Remark}
\newtheorem{conjecture}[theorem]{Conjecture}

\numberwithin{equation}{section}

\newcommand{\abs}[1]{\left|#1\right|}

\newcommand{\T}{\ensuremath{\mathbb{T}}}
\newcommand*{\R}{\ensuremath{\mathbb{R}}}

\newcommand*{\N}{\ensuremath{\mathbb{N}}}

\newcommand{\eps}{\varepsilon}

\renewcommand{\MR}[1]{} 

\usepackage{color, graphicx}
\usepackage{mathrsfs, dsfont}

\usepackage[]{hyperref}
\hypersetup{
    colorlinks=true,       
    linkcolor=red,          
    citecolor=blue,        
    filecolor=red,      
    urlcolor=cyan           
}

\def\dist{\mathop{\rm dist}\nolimits}    
\def\div{\mathop{\rm div}\nolimits}    
\def\curl{\mathop{\rm curl}\nolimits}    
 
\def\spt{\mathop{\rm Spt}\nolimits}

\newcommand{\be}{\begin{equation}}
\newcommand{\ee}{\end{equation}}

\title{Quantitative enstrophy bounds for measure vorticities}

\author[L. De Rosa]{Luigi De Rosa}
\address[L. De Rosa]{Gran Sasso Science Institute, viale Francesco Crispi, 7, 67100 L’Aquila, Italy}
\email{luigi.derosa@gssi.it}

\author[M. Marcotullio]{Margherita Marcotullio}
\address[M. Marcotullio ]{Gran Sasso Science Institute, viale Francesco Crispi, 7, 67100 L’Aquila, Italy}
\email{margherita.marcotullio@gssi.it}

\date{\today}


\begin{document}

\begin{abstract}
We consider the two-dimensional incompressible Navier--Stokes equations with measure initial vorticity. By means of improved Nash inequalities, we establish quantitative estimates for the enstrophy  depending on the absolute vorticity decay on balls.  The bounds are optimal in several aspects and yield to a conjecturally sharp rate of the dissipation in the Delort's class.
\end{abstract}

\maketitle

\section{Introduction}
On $\T^2\times [0,\infty)$, we consider the two-dimensional incompressible Navier--Stokes equations which, when written for the vorticity $\omega^\nu:=\curl u^\nu$, read as
\begin{equation}\label{NS-Vort} \tag{NS-Vort}
\left\{\begin{array}{ll}
\partial_t \omega^\nu +u^\nu \cdot \nabla \omega^\nu =\nu \Delta \omega^\nu \\
\curl u^\nu=\omega^\nu \\
\omega^\nu(\cdot, 0)=\omega^\nu_0,
\end{array}\right.
\end{equation}
where $\omega^\nu_0:=\curl u^\nu_0$ for a given divergence-free initial velocity $u_0^\nu$. In \eqref{NS-Vort} it is tacitly assumed that $u^\nu$ is the incompressible vector field such that $\curl u^\nu =\omega^\nu$, unique among the ones with a fixed spatial average.  By compatibility, the spatial average of $u^\nu$ is set to be the same of $u^\nu_0$ for all positive times.

For any $\nu>0$ and any divergence-free initial condition $u^\nu_0 \in L^2(\T^2)$,  global-in-time weak solutions 
$$
u^\nu \in L^\infty ([0,\infty); L^2(\T^2))\cap L^2([0,\infty); \dot{H}^1(\T^2))
$$ are known to exist since the seminal work of Leray \cite{L34}, and also Hopf \cite{Hopf51}. In two space dimensions they are unique \cites{RR,BV22}, they instantaneously become smooth, and they satisfy the energy identity
\begin{equation}
    \label{NS en bal}
    \frac{1}{2}\| u^\nu (t)\|_{L^2}^2 + \nu \int_0^t \|\omega^\nu (\tau)\|^2_{L^2}\,d\tau= \frac{1}{2}\| u^\nu_0\|_{L^2}^2\qquad \forall t\in [0,\infty).
\end{equation}
At least on the whole space $\R^2$, the two-dimensional Navier--Stokes equations are well-posed for any $\omega^\nu_0\in \mathcal M(\R^2)$ as well \cite{GG05}. This is more delicate with respect to the more classical setting of Leray since it also allows initial velocities with infinite kinetic energy\footnote{A finite energy initial velocity does not allow for Dirac masses in the vorticity \cite{delort1991existence}.}.

\subsection{Main results}
We are mainly interested in establishing estimates
for the enstrophy $\|\omega^\nu(t)\|^2_{L^2}$ when the initial vorticities are finite Borel measures. When $\{\omega^\nu_0\}_\nu\subset \mathcal M (\T^2)$ is bounded, the estimate 
\begin{equation}\label{enstroph trivial bound}
\|\omega^\nu(t)\|^2_{L^2}\lesssim \frac{1}{\nu t}\qquad \forall \nu, t>0
\end{equation}
is well-known. The same holds under the orthogonal assumption that $\{u^\nu_0\}_\nu\subset L^2 (\T^2)$ is bounded. We refer, for instance, to \cite{DRP24} for more details.

Our goal is to improve on \eqref{enstroph trivial bound}, possibly in a sharp way,  under the additional assumption that the absolute vorticity on balls, i.e.
\begin{equation}
        \label{vort on balls}
     \mathbb{M}_{\omega} (r):=\sup_{\substack{x \in \T^2  \\ \nu, t>0}} \int_{B_r(x)}|\omega^\nu(y,t)|\,dy \qquad \forall r>0,
    \end{equation}
    vanishes in the limit as $r\rightarrow 0$, uniformly in time and viscosity. Building on the strategy introduced in \cites{CLLS16,ELL25,LMP21}, in \cref{S:enstroph bounds} we present how  enstrophy bounds can be obtained from the knowledge of $\mathbb{M}_\omega (r)$. Although the precise expression might be implicit in general, here we state the main implication of the approach in a couple of specific cases.

\begin{theorem}\label{T:main}
    Let $\{u^\nu_0\}_\nu\subset L^2(\T^2)$ be such that $\{\omega^\nu_0\}_\nu\subset \mathcal M (\T^2)$ is bounded. Let $\{\omega^\nu\}_\nu$ be the corresponding solutions to \eqref{NS-Vort} and let $\mathbb{M}_\omega$ be defined as in \eqref{vort on balls}. 
    \begin{itemize}
        \item[$(a)$] If $\mathbb{M}_\omega(r)\lesssim r^\alpha$ for some $\alpha\in (0,2)$, then 
        $$
        \|\omega^\nu (t)\|^2_{L^2}\lesssim \frac{1}{(\nu t)^{\frac{2-\alpha}{2}}}  \qquad  \text{and}\qquad   \nu \int_0^T \|\omega^\nu (\tau)\|^2_{L^2}\,d\tau\lesssim (\nu T)^{\frac{\alpha}{2}}\qquad  \forall \nu>0,
$$
for all times $t,T>0$ such that $\nu t<1$ and $\nu T<1$, with the implicit constants independent of $\nu,t,T$.
     \item[$(b)$] If $\mathbb{M}_\omega(r)\lesssim |\log r|^{-\sfrac12}$ then 
     \begin{equation}\label{time bound for delort}
      \|\omega^\nu (t)\|^2_{L^2}\lesssim \frac{1}{\nu t \sqrt{|\log (\nu t)|}}  \qquad   \forall \nu>0,
          \end{equation}
     and for all times $t>0$ such that $\nu t<1$, with the implicit constant that does not depend on $\nu, t$.
     In particular, for any $\delta>0$, there holds
     $$
     \nu \int_\delta^T \|\omega^\nu (\tau)\|^2_{L^2}\,d\tau\lesssim \frac{\log \frac{T}{\delta}}{\sqrt{|\log (\nu T) |}} \qquad \forall \nu>0,
     $$
     and for any $T>\delta$ such that $\nu T<1$, with the implicit constant independent of $\nu,\delta,T$.
    \end{itemize}
\end{theorem}
The restriction to time scales below $\nu^{-1}$ is natural since otherwise the sharp estimate would be provided by the trivial bound \eqref{enstroph trivial bound}. We emphasize that Theorem \ref{T:main} does not require the sequence of initial velocities to stay bounded in $L^2(\T^2)$. In fact, the assumption  $\{u^\nu_0\}_\nu\subset L^2(\T^2)$ is most likely useless (see Remark \ref{R:infinite energy}). As an immediate corollary, we deduce lower bounds on the dissipation timescale. 
\begin{corollary}\label{C:time scale}
    Let $\{u^\nu_0\}_\nu\subset L^2(\T^2)$ be such that $\{\omega^\nu_0\}_\nu\subset \mathcal M (\T^2)$ is bounded. Let $\{\omega^\nu\}_\nu$ be the corresponding solutions to \eqref{NS-Vort} and let $\mathbb{M}_\omega$ be defined as in \eqref{vort on balls}. Let $\{T_\nu\}_\nu$ be a sequence of positive times. 
     \begin{itemize}
        \item[$(a')$] If $\mathbb{M}_\omega(r)\lesssim r^\alpha$ for some $\alpha\in (0,2)$, then 
        $$
       \lim_{\nu\rightarrow 0}  \frac{ T_\nu}{\nu^{-1}}=0 \qquad  \Longrightarrow \qquad  \lim_{\nu\rightarrow 0} \nu \int_0^{T_\nu} \|\omega^\nu (\tau)\|^2_{L^2}\,d\tau=0.
$$
     \item[$(b')$] If $\mathbb{M}_\omega(r)\lesssim |\log r|^{-\sfrac12}$ and, in addition, $\{u^\nu_0\}_\nu\subset L^2(\T^2)$ is strongly compact, for any  $\kappa\in (0,\frac12)$ it holds
      $$
       \lim_{\nu\rightarrow 0} \frac{ T_\nu}{e^{|\log \nu|^\kappa}}=0 \qquad  \Longrightarrow \qquad  \lim_{\nu\rightarrow 0} \nu \int_0^{T_\nu} \|\omega^\nu (\tau)\|^2_{L^2}\,d\tau=0.
$$
    \end{itemize}
\end{corollary}
In fact, the above results are not specific to the Navier--Stokes equations, since they apply to any advection diffusion equation with a divergence-free drift (see Remark \ref{R:general AD eq}). The rest of the introduction is devoted to describe the main context in which these results fit, together with the main improvements over the existing literature. We also discuss, and in some cases prove, their optimality.

\subsection{Main context and related literature}
Vorticity concentration relates to the global existence of weak solutions to the two-dimensional Euler equations with measure initial vorticity. This was first noted by Delort \cite{delort1991existence} who proved global existence for a finite energy initial velocity as soon as the singular\footnote{Throughout the whole paper, ``singular measure'' is always understood with respect to the Lebesgue measure.} part of the vorticity has distinguished sign. After more than 30 years, and several contributions \cites{scho95,evans1994hardy,LLX01,vecchi19931,majda1993remarks,CLLV19,ILL20,lant23,Tad01} by many authors, the type of initial data considered by Delort essentially remains the largest class for which global existence of weak solutions is known. Previous results were obtained in the seminal works by DiPerna and Majda \cites{DM87,DM88,diperna1987concentrations}, in which several tools for the study of ``loss of compactness'' issues in nonlinear PDEs have been developed.

The main observation by Delort \cite{delort1991existence} is that the $L^2_{\rm loc} (\T^2\times [0,\infty))$ strong compactness of the approximating velocity fields is not essential to obtain the global existence of a weak solution to the Euler equations. This is due to the special structure of the nonlinearity. He  established an abstract convergence result for time-dependent divergence-free vector fields as soon as the sequence of vorticities does not display spatial concentrations, uniformly in time. In our notation \eqref{vort on balls}, this reads as $\mathbb{M}_\omega (r)\rightarrow 0$ as $r\rightarrow 0$. Since the initial velocity has finite kinetic energy, the initial absolute vorticity does not concentrate. Then, the sign restriction on the singular part of the initial vorticity gives an effective, and the only currently known, way to propagate the non-concentration in time. More precisely, and as originally noted by Majda \cite{majda1993remarks}, the quantitative logarithmic decay
\begin{equation}\label{vort log decay}
    \mathbb{M}_\omega (r)\lesssim \frac{1}{\sqrt{|\log r|}}\qquad \forall r>0,
\end{equation}
can be proved\footnote{At least if the absolutely continuous part of the initial vorticity belongs to $L^p$ with $p>1$. The decay depends otherwise on the equi-integrability of the absolutely continuous part as well \cite{scho95}*{Theorem 3.6}.}. We remark that having $\{u^\nu\}_\nu\subset L^\infty([0,\infty);L^2(\T^2))$ bounded is essential to obtain \eqref{vort log decay}. Later on, Schochet \cite{scho95} made the above mechanisms even more transparent. Although simulations do not indicate concentrations \cites{lant23,LMP21,Kras87},  what happens for measure initial vorticities with no distinguished sign remains a formidable open problem in mathematical fluid dynamics. In this context, the strong $L^2_{\rm loc} (\T^2\times [0,\infty))$ compactness of the sequence of velocity is not expected, and kinematically false \cite{DS24}. The latter answers to a question raised by DiPerna and Majda \cite{diperna1987concentrations}, showing that the decay \eqref{vort log decay} is not even enough to deduce finiteness of the kinetic energy.

More recently, vorticity concentration has been shown to be connected to the so-called anomalous dissipation phenomenon. A sequence of solutions $\{\omega^\nu\}_\nu$ to \eqref{NS-Vort} is said to display ``no anomalous dissipation'' if 
\begin{equation}
    \label{no anom diss}
    \lim_{\nu\rightarrow 0} \nu \int_0^T \|\omega^\nu(\tau)\|^2_{L^2}\,d\tau =0.
\end{equation}
In \cite{DRP24}, the first author and Park proved that no vorticity concentration implies \eqref{no anom diss}. This improved on \cite{CLLS16}, where \eqref{no anom diss} was obtained under the assumption that $\{\omega^\nu_0\}_\nu\subset L^p(\T^2)$ is bounded for some $p>1$, providing, to the best of our knowledge, the very first ``Onsager's superctical'' energy conservation result for physically realizable weak solutions to the incompressible Euler equations in two space dimensions. We refer to the discussions in \cites{CLLS16,DI24,Drivas26,Shv18,Is24} for more details on Onsager's critical regularity and its connection to fluid turbulence. Later on, the relevance of the strong compactness of the sequence of velocities was reported in \cite{LMP21} (see also \cite{JLLL25} for the forced case). As already said before, we emphasize that the spatial non-concentration of the vorticity, even in the quantitative form \eqref{vort log decay}, is not enough to deduce the strong $L^2_{\rm loc} (\T^2\times [0,\infty))$ compactness for the velocity. 

The main result of \cite{DRP24} was then obtained by Elgindi--Lopes--Lopes \cite{ELL25}, independently. Their proof follows the ``Gagliardo--Niremberg  \& superquadratic Gr\"onwall'' strategy introduced in \cite{CLLS16}. We will recall it in \cref{S:enstroph bounds}, since it is also at the core of our approach. The key idea in \cite{ELL25} is to notice that a, uniform in viscosity and in time, decay of the absolute vorticity on balls allows for an ``improved Nash inequality''.  For instance, under the assumption \eqref{vort log decay}, in \cite{ELL25} the inequality
\begin{equation}
    \label{improved Nash from ELL}
    \| \omega^\nu (t)\|^2_{L^2}\lesssim \frac{\Vert \nabla \omega^\nu (t) \Vert_{L^2}}{\sqrt[4]{1+\abs{\log \Vert \nabla  \omega^\nu (t) \Vert_{L^2}}}}
\end{equation}
was established. Being quantitative, this approach allows to achieve \eqref{no anom diss}, at least for any positive time $\delta>0$, with the explicit rate\footnote{The implicit constant in \eqref{ELL rate} depends on $\delta$ and $T$, but it is otherwise independent of $\nu$.}
\begin{equation}
    \label{ELL rate}
    \nu \int_\delta^T \|\omega^\nu(\tau)\|^2_{L^2}\,d\tau\lesssim \frac{1}{\sqrt[4]{|\log \nu|}}\qquad \forall \nu >0,
\end{equation}
as soon as the initial vorticity can be split into an $L^p(\T^2)$ part, with $p>1$, and a nonnegative measure. Very recently, all the aforementioned results have been generalized in \cite{DRP25} and put in a more ``measure theoretic'' language. Moreover, the same rate \eqref{ELL rate} was reported again \cite{DRP25}*{Theorem 1.9}.

In the current work we improve on \eqref{ELL rate}. More precisely, as a  direct consequence of Theorem \ref{T:main}  (b), as well as \cite{DRP25}*{Proposition 5.1}, we obtain the following.

\begin{corollary}
    \label{C:rate for delort}
     Let  $\{u_0^\nu\}_{\nu}\subset L^2(\T^2)$ and $\{\omega_0^\nu\}_{\nu}\subset \mathcal M (\T^2)$ be bounded sequences in their respective norms. Let $p>1$ and assume that $\omega_0^\nu=f^\nu_0 +\mu_0^\nu$ with  $\mu^\nu_0\geq 0$ and $\{f^\nu_0\}_{\nu}\subset L^p(\T^2)$ bounded. Let $\{\omega^\nu\}_{\nu}$ be the corresponding sequence of Leray solutions to \eqref{NS-Vort}. For any $\delta>0$, it holds
     \begin{equation}\label{eq:rate for delort}
     \nu \int_\delta^T \|\omega^\nu (\tau)\|^2_{L^2}\,d\tau\lesssim \frac{\log \frac{T}{\delta}}{\sqrt{|\log (\nu T) |}} \qquad \forall \nu>0,
         \end{equation}
         and any $T>\delta$ such that $\nu T<1$, with the implicit constant independent of $\nu,\delta,T$. In particular, if in addition $\{u_0^\nu\}_{\nu}\subset L^2(\T^2)$ is strongly compact and $\kappa\in (0,\frac12)$ is any number, we have
\begin{equation}\label{timescale delort}
       \lim_{\nu\rightarrow 0} \frac{ T_\nu}{e^{|\log \nu|^\kappa}}=0 \qquad  \Longrightarrow \qquad  \lim_{\nu\rightarrow 0} \nu \int_0^{T_\nu} \|\omega^\nu (\tau)\|^2_{L^2}\,d\tau=0.
\end{equation}
\end{corollary}
The only use of the strong compactness of the initial data in proving \eqref{timescale delort} is to rule out dissipation in arbitrarily short time intervals\footnote{See for instance \cite{DRP25}*{Proposition 3.1}.} $(0,\delta)$. Such assumption is automatically satisfied if the vorticities decay algebraically on balls \cite{lant23}*{Corollary 2.15}. The claim \eqref{timescale delort} gives a lower bound on the timescale which is able to deplete the kinetic energy. This improves on \cite{DRP25}*{Theorem 1.9}, since any timescale $T_\nu\sim |\log \nu|^\beta$ is consistent with the antecedent in \eqref{timescale delort} for any $\beta>0$, no matter how large. 

Corollary \ref{C:time scale} (a') shows that  vorticity measures with an algebraic decay on balls need at least an amount of time $T_\nu\gtrsim \nu^{-1}$ in order to deplete the kinetic energy, independently on $\alpha\in (0,2)$. This generalizes the bound obtained in \cite{CLLS16} for $L^p(\T^2)$ vorticities, $p>1$. Thus, at the level of the timescale which is able to retain a nontrivial dissipation uniformly in the viscosity, there is no difference between $L^p(\T^2)$, for any $p>1$, and (possibly singular) measures, as soon as the latter have an algebraic decay on balls $\mathbb{M}_\omega (r)\lesssim r^\alpha$, for some $\alpha\in (0,2)$. On the other hand, there is a difference in the vanishing rate of the dissipation for a fixed time $T$. Indeed,  for a bounded sequence $\{\omega_0^\nu\}_\nu\subset L^p(\T^2)$, it was proved in \cite{CLLS16} that\footnote{The rate \eqref{rate for Lp} can be easily shown to be sharp on $\R^2$ by appropriately rescaling a  radially symmetric solution to the heat equation.}
\begin{equation}\label{rate for Lp}
     \nu \int_0^T \|\omega^\nu (\tau)\|^2_{L^2}\,d\tau\lesssim (\nu T)^{\frac{2(p-1)}{p}} \qquad \forall \nu, T>0.
\end{equation}
By the H\"older inequality it is clear that $\{\omega_0^\nu\}_\nu\subset L^p(\T^2)$ bounded implies 
$$
\mathbb{M}_\omega (r)\lesssim r^{\alpha_p}\qquad \forall r>0,\qquad \text{with }\, \alpha_p:=\frac{2(p-1)}{p}.
$$
If we apply  Theorem \ref{T:main} (a) to get a rate, we would then obtain $\lesssim \nu^{\sfrac{(p-1)}{p}}$, worsening \eqref{rate for Lp} by a square root. Nonetheless, as we shall discuss below, this is unavoidable and all the above results/considerations are essentially optimal.

\subsection{Sharpness of the bounds} It is natural to wonder about the sharpness of the results we have obtained here. We start discussing the case of Theorem \ref{T:main} (a). For simplicity, we work on the whole space $\R^2$. It is much easier to handle radially symmetric vorticities since, in this case, $u^\nu\cdot \nabla \omega^\nu\equiv 0$ and  \eqref{NS-Vort} reduces to the heat equation. Then, restricting to $\alpha=1$ is the most natural choice since it allows to consider the one-dimensional Hausdorff measure on the unit circle as an initial vorticity. It can be proved that this saturates the rate obtained in Theorem \ref{T:main} (a).
\begin{proposition}
    \label{P:saturate algebraic rate}
    There exists a compactly supported divergence-free $u_0\in L^\infty(\R^2)$ with $\omega_0\in \mathcal M(\R^2)$, such that, denoting by $\{\omega^\nu\}_\nu$ the corresponding solutions to \eqref{NS-Vort} on $\R^2\times [0,\infty)$, it holds 
    \begin{equation}
        \label{1d decay on balls}
        \mathbb{M}_\omega (r)\lesssim r\qquad \forall r>0
    \end{equation}
 and
 \begin{equation}
     \label{1d rate saturation}
     \nu\int_0^1\|\omega^\nu (\tau)\|^2_{L^2}\,d\tau\gtrsim \sqrt{\nu}\qquad \forall\, 0<\nu\ll 1.
 \end{equation}
\end{proposition}
We believe that, besides giving a simple scenario in which it is  possible to control the Navier--Stokes dynamics uniformly in the viscosity, there is nothing special about $\alpha=1$ with respect to any other value. In particular, and with some extra effort, it should be possible to prove the optimality of the rate obtained in Theorem \ref{T:main} (a) for any $\alpha\in (0,2)$.

Regarding the dissipative timescale, the one obtained in Corollary  \ref{C:time scale} (a') is optimal. Indeed, it is a simple exercise to show that the kinetic energy always depletes for times $T_\nu\gtrsim \nu^{-1}$, no matter the assumption on the initial data (see for instance \cite{DRP25}*{Remark 5.2}). The time $T_\nu\sim \nu^{-1}$ corresponds to the diffusive timescale dictated by the heat equation. 

We were not able to provide an example saturating Theorem \ref{T:main} (b), or its more effective version from Corollary \ref{C:rate for delort}. Although it is certainly possible to find vorticities such that $\mathbb{M}_\omega (r)\sim |\log r|^{-\sfrac12}$, the most natural ones do not seem to saturate the corresponding rate of the dissipation. We collect a couple of these ``failed attempts'' in \cref{S:failed tries}, with the hope that they are instructive. They both have, for very different reasons, a dissipation of the order $\lesssim |\log \nu|^{-1}$, thus missing \eqref{eq:rate for delort} by a square. More precisely, one of the examples shows the following.

\begin{proposition}\label{P:fake delort sharp}
    There exists a sequence $\{u^\nu_0\}_\nu$ of smooth compactly supported divergence-free vector fields  such that $\{u^\nu_0\}_\nu\subset L^2(\R^2)$ and $\{\omega^\nu_0\}_\nu\subset \mathcal M(\R^2)$ are bounded. Moreover,  $\omega^\nu_0 = f^\nu_0+\mu^\nu_0$ with $\{f^\nu_0\}_\nu \subset L^2(\R^2)$ bounded and $\mu_0^\nu\geq 0$. In addition, the corresponding sequence of Leray solutions to \eqref{NS-Vort} on $\R^2\times [0,\infty)$ enjoys
    $$
     \|\omega^\nu (t)\|^2_{L^2}\sim \frac{1}{\nu |\log \nu  |}\qquad \forall\, 0<\nu\ll 1,\, t\in (0,1).
    $$
\end{proposition}

Since, as we shall see in \cref{S:enstroph bounds}, the improved Nash inequality \eqref{improved Nash from ELL} is nearly the only tool used to obtain \eqref{eq:rate for delort}, a preliminary question is about the sharpness of the inequality itself. It turns out that \eqref{improved Nash from ELL} is optimal (see Proposition  \ref{P:nash ineq sharp}). However, the example we use to saturate \eqref{improved Nash from ELL} is a measure  concentrating on a rather sparse Cantor set with ``logarithmically zero'' Hausdorff dimension. The lack of control on the Navier--Stokes dynamics from such kind of  measure initial vorticity prevents us to use it to provide a reasonably sharp lower bound on its  dissipation. Summing up, the optimality of the Nash inequality \eqref{improved Nash from ELL} together with the sharpness proved in Proposition \ref{P:saturate algebraic rate} in the case of a linear decay of the measure on balls, seem to safely suggest the following. 

\begin{conjecture}\label{Conj}
 Let $\delta\in (0,1)$ and $p>1$. There exists  a divergence-free $u_0\in L^2(\T^2)$ with $\omega_0\in\mathcal M (\T^2)$,   $\omega_0=f_0 +\mu_0$ for some  $\mu_0\geq 0$ and $f_0\in L^p(\T^2)$, whose corresponding sequence $\{\omega^\nu\}_{\nu}$ of Leray solutions to \eqref{NS-Vort} enjoys
  $$
      \nu \int_\delta^1 \|\omega^\nu (\tau)\|^2_{L^2}\,d\tau\gtrsim \frac{1}{\sqrt{|\log \nu  |}}\qquad \forall\, 0<\nu\ll 1.
        $$   
\end{conjecture}

If one believes in the optimality of \eqref{time bound for delort}, which is the scenario suggested by the sharpness of the improved Nash inequality from Proposition  \ref{P:nash ineq sharp}, the above conjecture should not hold for $\delta=0$. What is certainly true is that, in the case a better bound exists, it cannot be proved with the same strategy we have used here. Aiming to a broader view, a better understanding of Conjecture \ref{Conj} might shield some light on the dynamics of the solutions constructed by Delort \cite{delort1991existence}, for which nothing seems to be known.

\bigskip

\section{Improved Nash inequalities}\label{S:nash ineq}
Let $f:\T^2 \rightarrow \R$ be a sufficiently smooth function with zero-average\footnote{This is only needed to rule out nontrivial constant functions, for which inequalities like \eqref{classical nash} are clearly false. }. The classical Nash inequality \cite{Nash58} reads as 
\begin{equation}
    \label{classical nash}
    \|f\|^2_{L^2}\lesssim \|f\|_{L^1}\|\nabla f\|_{L^2},
\end{equation}
for a geometric implicit constant. We are mainly interested in the regime where $\|f\|_{L^1}\sim 1$ and $\|\nabla f\|_{L^2}\gg 1$. This means that, for our purposes, the higher the power on the left-hand-side in \eqref{classical nash} the better. More generally, we are interested in replacing the square on $\|f\|_{L^2}$ with a superquadratic\footnote{Meaning that, at infinity, it diverges faster than a parabola.} function under suitable assumptions on $f$. The perhaps easiest case is when $\|f\|_{L^p}\sim 1$ for some $p>1$. In this setting, the well-known Gagliardo--Nirenberg inequality yields to 
$$
\|f\|^{\frac{2}{2-p}}_{L^2}\lesssim  \|\nabla f\|_{L^2} \qquad \text{if } p\in (1,2).
$$
The case where $\|f\|_{L^1}\sim 1$ is the only available bound is more interesting. We set 
$$
\mathbb{M}_f (r):=\sup_{x \in \T^2} \int_{B_r (x)}|f(y)|\,dy \qquad \forall r>0.
$$
Whenever $\mathbb{M}_f (r)\rightarrow 0$ as $r\rightarrow 0$, it is possible to improve on \eqref{classical nash}. Moreover, the corresponding inequality is uniform among all functions with the same decay on balls and the same total mass. To the best of our knowledge, this was first noted in \cite{ELL25}.

\begin{proposition}[\cite{ELL25}*{Proposition 3.2}]\label{P: nash superquadratic ELL}
Let $f\in W^{1,2}(\T^2)$ be with zero average. Assume that $\|f\|_{L^1}\leq M$ for a constant $M>0$ and $\mathbb{M}_f(r)\leq \phi (r)$ for all $r>0$, for a  function $\phi\in C^0([0,\infty))$ such that $\phi (r)\rightarrow 0$ as $r\rightarrow 0$. There exists an increasing superquadratic function $\Psi\in C^1([0,\infty))$, depending only on $M$ and $\phi$, such that 
$$
\Psi\left(\|f\|^2_{L^2}\right) \leq \|\nabla f\|^2_{L^2}.
$$
\end{proposition}
 
When\footnote{This would follow, for instance, from $f\in H^{-1}\cap \left(L^p + \mathcal M_{\geq 0}\right)$ for some $p>1$ \cite{scho95}*{Theorem 3.6}.} $\phi(r) = |\log r|^{-\sfrac{1}{2}}$, the inequality \eqref{improved Nash from ELL} was then obtained \cite{ELL25}*{Proposition 4.1}. The proof given in \cite{ELL25} uses Fourier analysis and Riesz--Thorin interpolation. Here we follow a more direct physical space based approach as described below.

Let $\{\rho_\ell\}_\ell\subset C^\infty_c (\R^2)$ be a family of Friedrichs mollifiers, $\ell\in (0,1)$. We split
\begin{align}
\|f\|^2_{L^2}&\lesssim \|f - f*\rho_\ell\|^2_{L^2}+ \| f*\rho_\ell\|^2_{L^2} \\
&\lesssim  \ell^2 \|\nabla f\|^2_{L^2} + \|f*\rho_\ell\|_{L^1}\|f*\rho_\ell\|_{L^\infty} \\
&\lesssim \ell^2 \|\nabla f\|^2_{L^2} + \|f\|_{L^1}\|f*\rho_\ell\|_{L^\infty}, 
\end{align}
for a geometric implicit constant. Furthermore, the last term can be estimated by 
$$
\|f*\rho_\ell\|_{L^\infty}=\sup_{x\in \T^2}\left| \int f(y)\rho_\ell (x-y)\,dy\right|\leq \frac{\mathbb{M}_f (\ell)}{\ell^2}.
$$
We have achieved 
\begin{equation}\label{main splitting mollif}
\|f\|^2_{L^2}\lesssim \ell^2 \|\nabla f\|^2_{L^2} + \|f\|_{L^1}\frac{\mathbb{M}_f (\ell)}{\ell^2}\qquad \forall \ell\in (0,1).
\end{equation}
Note that, if we now use the trivial bound $\mathbb{M}_f(\ell)\leq \|f\|_{L^1}$, we obtain 
$$
\|f\|^2_{L^2}\lesssim \ell^2 \|\nabla f\|^2_{L^2} + \frac{\|f\|^2_{L^1}}{\ell^2}\qquad \forall \ell\in (0,1).
$$
We can then optimize in $\ell$ to find\footnote{Note that, thanks to the zero average assumption on $f$, its gradient cannot be identically zero without $f$ being trivial.}
\begin{equation}
    \label{ell opt}
\ell^2_{\rm opt} = \frac{\| f\|_{L^1}}{\|\nabla f\|_{L^2}},
\end{equation}
from which the classical Nash inequality \eqref{classical nash} immediately follows. To be rigorous, we are in fact  allowed to make the choice \eqref{ell opt} only if $\ell_{\rm opt} <1$, i.e. when $\| f\|_{L^1}< \| \nabla f\|_{L^2}$. However, in the case $\| \nabla f\|_{L^2}\leq \| f\|_{L^1}$ the Nash inequality \eqref{classical nash} is a direct consequence of the Poincaré inequality
$$
\|f\|^2_{L^2}\lesssim \|\nabla f\|^2_{L^2}\lesssim   \| f\|_{L^1} \|\nabla f\|_{L^2}.
$$
In this simple procedure we have not used the additional smallness of $\mathbb{M}_f(\ell)\leq \phi (\ell)$. In this case, by also using  $\|f\|_{L^1}\leq M$, the optimal choice of $\ell$ would be such that 
$$
\frac{\ell^4_{\rm opt}}{\phi (\ell_{\rm opt}) }\sim \frac{1}{\|\nabla f\|^2_{L^2}}.
$$
Since $\phi(\ell)\rightarrow 0$ as $\ell\rightarrow 0$, this leads to Proposition \ref{P: nash superquadratic ELL}. Details are left to the reader. 

In the next proposition we summarize the main implication of the approach in the two cases we are mainly interested in. Again, we emphasize that the inequalities we are going to establish should be thought in the regime where $\|\nabla f\|_{L^2}$ is larger than a given universal constant since, otherwise, the Poincaré inequality would provide a better estimate.

\begin{proposition}
    \label{P:two nash ineq}
    Under the assumptions of Proposition \ref{P: nash superquadratic ELL}, the following hold.
    \begin{itemize}
        \item[$(1)$] If $\phi (r)\lesssim r^\alpha$ for some $\alpha\in (0,2)$, then 
        \begin{equation}
            \label{Nash algebraic}
            \|f\|_{L^2}^{\frac{4-\alpha}{2-\alpha}}\lesssim \|\nabla f\|_{L^2}.
        \end{equation}
        \item[$(2)$] If $\phi (r)\lesssim |\log r|^{-\sfrac{1}{2}}$, then
        \begin{equation}
            \label{Nash delort}
             \|f\|^2_{L^2} \lesssim \frac{   \|\nabla f\|_{L^2} }{\sqrt[4]{1+\left| \log \|\nabla f\|_{L^2}  \right|}}.
        \end{equation}
    \end{itemize}
\end{proposition}
\begin{proof}
    We prove the two separately. We emphasize that all the implicit constants in the estimates below only depend on the total mass of $f$, which we are assuming it satisfies the universal bound $\|f\|_{L^1}\leq M$, and on the  constant that might be present in the assumption about the vanishing rate of $\phi$. As it will be clear from the proof, the dependence can be made explicit if needed.
    
    \underline{\textsc{Proof of (1)}}.  By \eqref{main splitting mollif} we deduce
    $$
    \|f\|^2_{L^2}\lesssim \ell^2 \|\nabla f\|^2_{L^2} +\ell^{\alpha -2}\qquad \forall 
    \ell \in (0,1).
    $$
    If $\|\nabla f\|_{L^2}> 1$, the optimal choice is 
    $$
    \ell^{4-\alpha}_{\rm opt} =\frac{1}{\|\nabla f\|^2_{L^2}},
    $$
    immediately yielding to \eqref{Nash algebraic}. When $\|\nabla f\|_{L^2}\leq 1$, the same conclusion follows by the Poincaré inequality
    $$
    \|f\|^{\frac{4-\alpha}{2-\alpha}}_{L^2}\lesssim  \|\nabla f\|^{\frac{4-\alpha}{2-\alpha}}_{L^2} \lesssim \|\nabla f\|_{L^2}
    $$
    since $\frac{4-\alpha}{2-\alpha}>1$.

     \underline{\textsc{Proof of (2)}}.  By \eqref{main splitting mollif} we deduce
    $$
    \|f\|^2_{L^2}\lesssim \ell^2 \|\nabla f\|^2_{L^2} +\frac{1}{\ell^2\sqrt{|\log \ell|}}\qquad \forall 
    \ell \in (0,1).
    $$
    Assume for the moment that $\|\nabla f\|_{L^2}>e$ and choose 
    $$
    \ell^2_{\rm opt}=\frac{1}{\|\nabla f\|_{L^2} \sqrt[4]{1+\log \|\nabla f\|_{L^2} }}.
    $$
    Since 
    $$
    |\log \ell_{\rm opt}|\gtrsim \log \|\nabla f\|_{L^2} \qquad \text{and} \qquad \frac{\sqrt[4]{1+\log \|\nabla f\|_{L^2}}}{\sqrt{\log \|\nabla f\|_{L^2}}}\lesssim \frac{1}{\sqrt[4]{1+ \log \|\nabla f\|_{L^2} }}
    $$
    for some implicit universal constants, such choice yields to 
    \begin{align}
         \|f\|^2_{L^2}&\lesssim \frac{\|\nabla f\|_{L^2}}{ \sqrt[4]{1+ \log \|\nabla f\|_{L^2} }} + \frac{\|\nabla f\|_{L^2} \sqrt[4]{1+ \log \|\nabla f\|_{L^2} }}{\sqrt{|\log \ell_{\rm opt}|}} \\
         & \lesssim \frac{\|\nabla f\|_{L^2}}{ \sqrt[4]{1+ \log \|\nabla f\|_{L^2}  }} + \frac{\|\nabla f\|_{L^2} \sqrt[4]{1+ \log \|\nabla f\|_{L^2}  }}{\sqrt{\log \|\nabla f\|_{L^2}}} \\
         &\lesssim  \frac{\|\nabla f\|_{L^2}}{ \sqrt[4]{1+ \log \|\nabla f\|_{L^2}  }}, 
    \end{align}
    that is the desired inequality. As already discussed several times, the case $\|\nabla f\|_{L^2}\leq e$ is simpler and it follows by the Poincaré inequality.
\end{proof}

\begin{remark}
    All the above considerations  apply to any dimension. If $d\geq 2$, the corresponding classical Nash inequality on $\T^d$ becomes 
    $$
    \|f\|^2_{L^2}\lesssim \|f\|^{\frac{4}{2+d}}_{L^1}\|\nabla f\|^{\frac{2d}{2+d}}_{L^2},
    $$
    which can be analogously improved under the assumption $\mathbb{M}_f(r)\rightarrow 0$ as $r\rightarrow 0$.
\end{remark}

Both the inequalities established in Proposition \ref{P:two nash ineq} are optimal. The examples saturating them are given by bounded sequences $\{f_n\}_n\subset L^1(\T^2)$ concentrating on a lower-dimensional set. For \eqref{Nash algebraic} the set will have Hausdorff dimension equal to $\alpha$, while for \eqref{Nash delort} we will need a set with ``logarithmically zero'' Hausdorff dimension. Modulo the specific choices of the parameters, the construction is the same in both cases, producing a self-similar Cantor set.
 
\begin{proposition}
    \label{P:nash algebraic sharp}
    Let $\alpha\in (0,2)$. There exists a sequence $\{f_n\}_n\subset W^{1,\infty}(\T^2)$ of zero-average functions, bounded in $L^1(\T^2)$, such that 
   $$
   \sup_{n\geq 1} \mathbb{M}_{f_n} (r)\lesssim r^\alpha\qquad \forall r\in (0,1)
   $$
 and satisfying  
    \begin{equation}\label{eq algebraic saturate}
    \lim_{n\rightarrow \infty} \|f_n\|_{L^2} = \infty \qquad \text{and}\qquad 0<\liminf_{n\rightarrow\infty}S_n\leq \limsup_{n\rightarrow\infty}S_n<\infty,
    \end{equation}
  where $S_n:=\frac{\|f_n\|_{L^2}^{\frac{4-\alpha}{2-\alpha}}}{\|\nabla f_n\|_{L^2}}$.
\end{proposition}

In the proofs below we will use the symbol $\sim$ to denote quantities with the same asymptotic behavior. More precisely, if $\{a_n\}_n$ and $\{b_n\}_n$ are two sequences of nonnegative real numbers, we write $a_n\sim b_n$ if 
$$
0<\liminf_{n\rightarrow \infty} \frac{a_n}{b_n}\leq \limsup_{n\rightarrow \infty} \frac{a_n}{b_n}<\infty.
$$

\begin{proof}
   Let $Q:=(0,1)^2\subset \R^2$. In order to define the sequence $\{f_n\}_n$, we need to construct a nested family of sets $\{C_n\}_n\subset Q$. For any $n\geq 1$, the set $C_n$ will be the $n$-th iteration of a self-similar Cantor set. More precisely, if $\delta_n\in (0,1)$ is a small parameter to be defined, the set $C_n$ consists in the union of $N_n=4^n$ pairwise disjoint disks of radius $\delta_n$ (see \cref{cantor fig}), i.e. 
   $$
   C_n:=\bigcup_{i=1}^{N_n} B_{\delta_n} (x_i).
   $$
   	\begin{figure}
		\centering
			\includegraphics[width=0.5\textwidth]{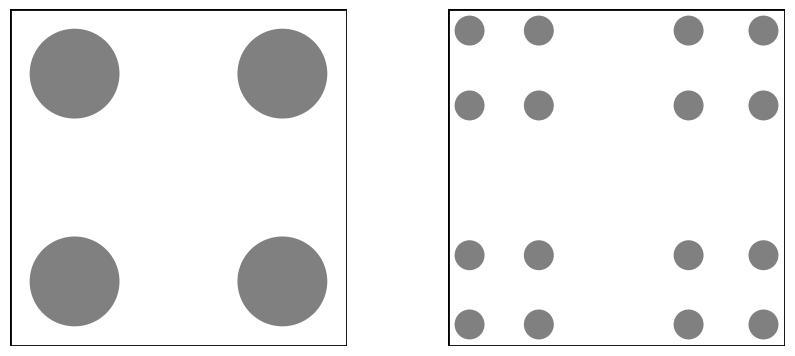} 
					\caption{Iterations $n=1$ (on the left) and $n=2$ (on the right) of the self-similar Cantor set.}\label{cantor fig}
	\end{figure}
   We define the nonnegative function $\tilde f_n\in W^{1,\infty}_c\left(Q\right)$ as
 \begin{equation}\label{def tilde fn}
 \tilde f_n(x) = \begin{cases}
\delta_n^{-2}N_n^{-1} & \text{if }\; x \in B_{\delta_n} (x_i)\\[4pt]
\delta_n^{-2}N_n^{-1} \Big( 1- \delta_n^{-1}\dist(x,\partial B_{\delta_n}(x_i))\Big) & \text{if }\; x \in B_{2\delta_n} (x_i)\setminus B_{\delta_n} (x_i)
\end{cases}
  \end{equation}
 for all $i=1,\dots,N_n$, so that $\spt \tilde f_n= \bigcup_{i=1}^{N_n} \overline B_{2\delta_n} (x_i)$. It follows that $\nabla \tilde f_n$ is supported in the union of all the shells and 
\begin{equation}
    \label{gradient pointwise estimate}
    |\nabla \tilde f_n (x)|=\delta_n^{-3}N_n^{-1} \qquad \forall x\in \bigcup_{i=1}^{N_n}B_{2\delta_n} (x_i)\setminus B_{\delta_n} (x_i).
    \end{equation}
Also
 $$
\int_Q \tilde f_n(x)\,dx=c_n\qquad \text{with} \qquad \{c_n\}_n\subset \ell^\infty(\N) \text{ bounded}.
 $$
 Moreover, being compactly supported in $Q$, the above procedure defines a periodic function. Therefore, setting $f_n:=\tilde f_n-c_n$ for all $n\geq 1$, we obtain a sequence of zero-average functions $\{f_n\}_n\subset W^{1,\infty}(\T^2)$ bounded in $L^1(\T^2)$. We now choose 
 $$
 \delta_n:= 4^{-\frac{n}{\alpha}}
 $$
 and check that $\{f_n\}_n$ satisfies all the required properties. From now on we will ignore universal multiplicative constants since they play no essential role. 

\underline{\textsc{Bound on balls}}: We want to show  $\sup_n \mathbb{M}_{f_n} (r)\lesssim r^\alpha$.

Since the presence of the additional constant $c_n$ would only add a lower order term, it is enough to prove the claim for $\tilde f_n$.

Assume that $r\leq \delta_n$. It is clear that in this case the largest mass of $\tilde f_n$ on $B_r(x)$ is achieved when $B_r(x)\subset B_{\delta_n}(x_i)$ for some $i=1,\dots,N_n$. Therefore 
$$
\mathbb{M}_{\tilde f_n} (r) \lesssim \delta_n^{-2}N_n^{-1} r^2.
$$
We want to check that the latter quantity is bounded by $r^\alpha$, uniformly in $n$. This is equivalent to check 
$$
r^{2-\alpha}\leq \delta_n^{2}N_n,
$$
which is clearly satisfied if $r\leq \delta_n$ since $\alpha\in (0,2)$.

Assume now $r>\delta_n$. Then, $r\in (\delta_{n-k+1},\delta_{n-k}]$ for some $k\in \{1,\dots,n\}$. Therefore, the worse bound on the mass of $\tilde f_n$ on $B_r(x)$ is when such ball touches $N_k$ among the disks $\{B_{\delta_n}(x_i)\}_i$. It follows that 
$$
\mathbb{M}_{\tilde f_n}(r) \lesssim N_k \delta_n^2 \delta_n^{-2}N_n^{-1} = N_{k-n}.
$$
The claim is proved if $N_{k-n}\lesssim r^\alpha$, which is equivalent to $r^{-\alpha}\lesssim N_{n-k}$. Since we are assuming $r>\delta_{n-k+1}$, we can bound 
$$
r^{-\alpha}< \delta_{n-k+1}^{-\alpha}\lesssim N_{n-k},
$$ 
concluding the proof of the claim.

\underline{\textsc{Saturating Nash}}: We want to show \eqref{eq algebraic saturate}.

We compute
$$
\|\tilde f_n\|^2_{L^2}\sim \left(\delta_n^{-2} N^{-1}_n\right)^2 \delta_n^2 N_n = 4^{\frac{2-\alpha}{\alpha}n}.
$$
Recalling that $f_n=\tilde f_n -c_n$ with $\{c_n\}_n$ bounded in $\ell^\infty(\N)$ and $\{\tilde f_n\}_n$  bounded in $L^1(\T^2)$, we obtain 
\begin{equation}
    \label{L2 norm}
    \| f_n\|^2_{L^2}= \|\tilde f_n\|^2_{L^2} +c_n^2 -2c_n \|\tilde f_n\|_{L^1}\sim  4^{\frac{2-\alpha}{\alpha}n}.
\end{equation}
By \eqref{gradient pointwise estimate} we also get 
\begin{equation}
    \label{H1 norm}
    \|\nabla f_n\|^2_{L^2}= \|\nabla \tilde f_n\|^2_{L^2} \sim \left(\delta_n^{-3}N_n^{-1}\right)^2\delta_n N_n\sim  4^{\frac{4-\alpha}{\alpha}n}.
\end{equation}
The claim \eqref{eq algebraic saturate} immediately follows by \eqref{L2 norm} and \eqref{H1 norm}.
 \end{proof}

\begin{proposition}
    \label{P:nash ineq sharp}
There exists a sequence $\{f_n\}_n\subset W^{1,\infty}(\T^2)$ of zero-average functions, bounded in $L^1(\T^2)$, such that 
$$
\sup_{n\geq 1} \mathbb{M}_{f_n} (r)\lesssim \frac{1}{\sqrt{|\log r|}} \qquad \forall r\in \left(0,\frac12\right)
$$
and satisfying
    \begin{equation}\label{eq delort saturate}
    \lim_{n\rightarrow \infty} \|f_n\|_{L^2} = \infty \qquad \text{and}\qquad 0<\liminf_{n\rightarrow\infty}S_n\leq \limsup_{n\rightarrow\infty}S_n<\infty,
    \end{equation}
    where $S_n:=\frac{\|f_n\|^2_{L^2}\sqrt[4]{\log \|\nabla f_n\|_{L^2}}}{\|\nabla f_n\|_{L^2}}$.
\end{proposition}
\begin{proof}
    The sequence $\{f_n\}_n$ is the same we have defined in the proof of Proposition  \ref{P:nash algebraic sharp}. However, in this case we need to choose 
    $$
    \delta_n:=e^{-4^{2n}}.
    $$
    We will again ignore universal multiplicative constants.

\underline{\textsc{Bound on balls}}: We want to show  $\sup_n \mathbb{M}_{f_n} (r)\lesssim |\log r|^{-\sfrac12}$.

As already done in the proof of Proposition  \ref{P:nash algebraic sharp}, it is enough to estimate $\tilde f_n$ only. If $r\leq \delta_n$, the largest mass of $\tilde f_n$ on $B_r(x)$ is achieved when $B_r(x)\subset B_{\delta_n}(x_i)$ for some $i=1,\dots,N_n$. Thus
$$
\mathbb{M}_{\tilde f_n} (r) \lesssim \delta_n^{-2} N_n^{-1}r^2.
$$
We want to check that $\delta_n^{-2}N_n^{-1}r^2\leq |\log r|^{-\sfrac12}$, which is equivalent to 
\begin{equation}
\label{condition to check}
r^2\sqrt{|\log r|}\leq \delta_n^2N_n.
\end{equation}
For sufficiently small values of $r$, the function $r^2\sqrt{|\log r|}$ is monotone increasing. Thus, from  $r\leq \delta_n$ we deduce 
$$
r^2\sqrt{|\log r|}\leq \delta_n^2\sqrt{|\log \delta_n|}= \delta_n^2\sqrt{4^{2n}}=\delta_n^2N_n,
$$
proving \eqref{condition to check}.

Assume now $r>\delta_n$. Then $r\in (\delta_{n-k+1},\delta_{n-k}]$ for some $k\in \{1,\dots, n\}$. In such case, the largest mass of $\tilde f_n$ that can be contained in a ball of radius $r$ is achieved when such ball touches $N_k$ disks among $\{B_{\delta_n}(x_i)\}_i$. Therefore 
$$
\mathbb{M}_{\tilde f_n} (r)\lesssim N_k\delta_n^2\delta_n^{-2} N_n^{-1} = N_{k-n}.
$$
On the other hand, from $r>\delta_{n-k+1}$, we also have 
$$
|\log r|^{-\frac12}>|\log \delta_{n-k+1}|^{-\frac12}\gtrsim N_{k-n}.
$$
The claim is proved.

\underline{\textsc{Saturating Nash}}: We want to show \eqref{eq delort saturate}.

As in the proof of Proposition  \ref{P:nash algebraic sharp} we have 
$$
\|f_n\|_{L^2}^2\sim \left(\delta_n^{-2}N_n^{-1}\right)^2\delta_n^2 N_n = \delta_n^{-2}N_n^{-1}
$$
and
$$
\|\nabla f_n\|_{L^2}^2\sim \left(\delta_n^{-3}N_n^{-1}\right)^2\delta_n^2 N_n = \delta_n^{-4}N_n^{-1}.
$$
Since $\delta_n=e^{-4^{2n}}$, these yield to 
$$
\|f_n\|^2_{L^2}\sim 4^{-n}e^{2\cdot 4^{2n}}\qquad \text{and}\qquad \|\nabla f_n\|_{L^2}\sim 4^{-\frac{n}{2}}e^{2\cdot 4^{2n}}.
$$
Therefore 
$$
\frac{\|f_n\|^2_{L^2}\sqrt[4]{\log \|\nabla f_n\|_{L^2}}}{\|\nabla f_n\|_{L^2}}\sim 4^{-\frac{n}{2}}\sqrt[4]{2 \cdot 4^{2n} -\frac{n}{2}\log 4}\sim \sqrt[4]{1-n 4^{-2n}},
$$
concluding the proof.
\end{proof}

\begin{remark}
    By \cite{lant23}*{Corollary 2.15}, the sequence $\{f_n\}_n$ constructed in Proposition  \ref{P:nash algebraic sharp} can be shown to be bounded in $H^{-1}(\T^2)$. On the other hand, we do not expect the same to be true for the one constructed in Proposition  \ref{P:nash ineq sharp}. That is because the  homogeneous negative Sobolev norm squared coming from a single $\tilde f_n\mathbbm{1}_{B_{2\delta_n}(x_i)}$ is of the order $N_n^{-2}|\log \delta_n|\sim 1$, for all $i=1,\dots,4^n$. By modifying the definition of $\tilde f_n$ from \eqref{def tilde fn} so that $\tilde f_n$ has zero average on $B_{2\delta_n} (x_i)$ for all $i$, the global negative Sobolev norm stays bounded. In fact, in doing so, one gets $\|\tilde f_n\|^2_{H^{-1}}\lesssim N_n^{-1}$.
\end{remark}

\section{Quantitative enstrophy bounds}\label{S:enstroph bounds}

In this section, we show how enstrophy bounds can be obtained from the knowledge of the absolute vorticity on balls $\mathbb{M}_\omega(r)$. We start by presenting the general strategy, and postpone the proof of Theorem \ref{T:main} later on.

We follow the strategy from \cites{CLLS16,LMP21,ELL25}, where the key idea is to get an effective Gr\"onwall-type estimate. All the computations below are justified since everything is smooth for  positive times. By \eqref{NS-Vort} we have 
\begin{equation}
    \label{enstrophy balance} 
    \frac{d}{dt}\Vert \omega^{\nu}(t)\Vert_{L^2}^2 = - 2 \nu \Vert \nabla \omega^{\nu}(t)\Vert_{L^2}^2\qquad \forall \nu, t>0.
\end{equation}
We now assume that 
$$
\Psi \left(\|\omega^{\nu}(t)\Vert_{L^2}^2\right)\leq \|\nabla \omega^{\nu}(t)\Vert_{L^2}^2
$$
for a monotone increasing superquadratic function $\Psi\in C^1$. Since $\|\omega^\nu(t)\|_{L^1}\leq \|\omega^\nu_0\|_{\mathcal M}$ and $\{\omega^\nu_0\}_\nu\subset \mathcal M(\T^2)$ is bounded, Proposition \ref{P: nash superquadratic ELL} makes the above assumption legitimate as soon as $\mathbb{M}_\omega(r)\rightarrow 0$ as $r\rightarrow 0$. Therefore
\begin{equation}
    \label{enstrophy inequality} 
    \frac{d}{dt}\Vert \omega^{\nu}(t)\Vert_{L^2}^2 \leq  - 2 \nu \Psi \left(\|\omega^{\nu}(t)\Vert_{L^2}^2\right) \qquad \forall \nu, t>0.
\end{equation}
For convenience we set 
\begin{equation}
    \label{z rescaing time}
    z(t):= \|\omega^{\nu}(\nu^{-1}t)\|_{L^2}^2,
\end{equation}
which then satisfies\footnote{With no loss of generality we are incorporating the constant $2$ in \eqref{enstrophy inequality} inside the superquadratic function $\Psi$.} $z'(t)\leq -\Psi (z(t))$ for all $t>0$. Let 
\begin{equation}\label{F def}
F(w):=\int_w^\infty \frac{1}{\Psi (v)}\,dv \qquad \forall w>0.
\end{equation}
Note that $F(w)<\infty$ for all $w>0$ since $\Psi$ is superquadratic. We now claim that 
\begin{equation}
    \label{claim comparison principle with F}
    F(z(t))\geq t \qquad \forall t>0.
\end{equation}
 The claim follows by a comparison principle. Since the differential inequality might be singular at $t=0$\footnote{This happens when $z(t)\rightarrow \infty$ as $t\rightarrow 0$, which is the typical case for measure initial vorticities.}, we provide the details. 

Let $\delta>0$. Since everything is smooth in $(\delta,\infty)$, by the comparison principle we get that the function $z$ stays below the solution $y_\delta:[\delta,\infty)\rightarrow (0,\infty)$ to 
\begin{equation}
    \label{ode for y delta}
   \begin{cases}
y'_\delta(t) = -\Psi (y_\delta(t))\\[4pt]
y_\delta(\delta)=z(\delta).
\end{cases}
\end{equation}
A solution to \eqref{ode for y delta} is given, implicitly, by 
$$
\int_{y_\delta(t)}^{z(\delta)} \frac{1}{\Psi (v)}\,dv=t-\delta\qquad \forall t\geq \delta.
$$
Then, since $z(t)\leq y_\delta (t)$ for all $t\geq \delta$, by letting 
$$
F_\delta(w):=\int_{w}^{z(\delta)} \frac{1}{\Psi (v)}\,dv,
$$
we get 
\begin{equation}
    \label{inequality with Fdelta monotone}
    t-\delta=F_\delta(y_\delta (t))\leq F_\delta(z(t))\qquad \forall t\geq \delta.
\end{equation}
Note that $z$ is monotone decreasing by \eqref{enstrophy balance}, and thus it admits a, possibly infinite, limit $L:=\lim_{\delta\rightarrow 0} z(\delta)$. Therefore, we can let $\delta\rightarrow 0$ in \eqref{inequality with Fdelta monotone} to obtain 
$$
t\leq \int_{z(t)}^L \frac{1}{\Psi(v)}\,dv \leq \int_{z(t)}^\infty \frac{1}{\Psi(v)}\,dv = F(z(t))\qquad \forall t>0.
$$
The claim \eqref{claim comparison principle with F} is thus proved. Recalling \eqref{z rescaing time}, we have achieved 
\begin{equation}
    \label{main general enstrophy bound}
    F\left(\|\omega^\nu(t)\|^2_{L^2}\right)\geq \nu t\qquad \forall \nu, t>0.
\end{equation}
We are now ready to prove Theorem \ref{T:main}.

\begin{proof}[Proof of Theorem \ref{T:main}]
    We prove the two claims separately. 

    \underline{\textsc{Proof of $(a)$}}. By Proposition \ref{P:two nash ineq} (1) we can choose $\Psi (v) = C v^\frac{4-\alpha}{2-\alpha}$ for some constant $C>0$. Thus, recalling the definition of $F$ from \eqref{F def}, by \eqref{main general enstrophy bound} we deduce 
    $$
    \frac{2}{2-\alpha} \frac{1}{\|\omega^\nu(t)\|^{\frac{4}{2-\alpha}}_{L^2}} = \int_{\|\omega^\nu(t)\|^2_{L^2}}^\infty \frac{1}{v^\frac{4-\alpha}{2-\alpha}}\,dv \gtrsim \nu t \qquad \forall \nu,t>0.
    $$
    Equivalently $\|\omega^\nu(t)\|^2_{L^2}\lesssim (\nu t)^\frac{\alpha-2}{2}$, from which we conclude 
    $$
    \nu\int_0^T \|\omega^\nu(\tau)\|^2_{L^2}\,d\tau\lesssim (\nu T)^\frac{\alpha}{2}\qquad \forall \nu,T>0.
    $$

\underline{\textsc{Proof of $(b)$}}. We want to prove that \eqref{main general enstrophy bound} implies 
\begin{equation}
    \label{enstroph bound for delort}
     \|\omega^\nu (t)\|^2_{L^2}\lesssim \frac{1}{\nu t \sqrt{|\log (\nu t)|}}  \qquad   \text{if }  \nu t<1.
\end{equation}
   For convenience, we reparametrize time by setting $s=\nu t$. Then, by \eqref{main general enstrophy bound} and the monotonicity of $F$, \eqref{enstroph bound for delort}  holds if 
    \begin{equation}
        \label{reformulating the bound for delort}
        \int^\infty_{\frac{1}{s\sqrt{|\log s|}}} \frac{1}{\Psi(v)}\,dv = F\left(\frac{1}{s\sqrt{|\log s|}}\right)\lesssim s\qquad \forall s\in (0,1).
    \end{equation}
Note that \eqref{reformulating the bound for delort} holds true for $s\rightarrow 1$ since, in such case, the left-hand-side will vanish while the right-hand-side goes to $1$. We can thus restrict to $s\in (0,s_1)$ for some $s_1<1$. Let $s_0 \in (0,s_1)$. For $s\in (s_0,s_1)$ the claim above is obvious for a sufficiently large constant $C=C(s_0,s_1)>0$ since 
$$
  \int^\infty_{\frac{1}{s\sqrt{|\log s|}}} \frac{1}{\Psi(v)}\,dv\leq  \int^\infty_{K} \frac{1}{\Psi(v)}\,dv  \leq \frac{s}{s_0} \int_{K}^\infty  \frac{1}{\Psi(v)}\,dv =: C s\qquad \forall s\in (s_0,s_1),
$$
where 
$$
K:=\frac{1}{\max \left(s_0\sqrt{|\log s_0|},s_1\sqrt{|\log s_1|}\right)}.
$$
Therefore, to conclude \eqref{reformulating the bound for delort} it is enough to show 
\begin{equation}\label{right derivative bounded}
\limsup_{s\rightarrow 0}\frac{1}{s}  \int^\infty_{\frac{1}{s\sqrt{\log  \frac{1}{s}}}} \frac{1}{\Psi(v)}\,dv<\infty.
\end{equation}
Note that \eqref{right derivative bounded} is nothing but the boundedness of the right derivative at $s=0$ of the map 
$$
s\mapsto  \int^\infty_{\frac{1}{s\sqrt{\log \frac{1}{s}}}} \frac{1}{\Psi(v)}\,dv.
$$
Therefore, \eqref{right derivative bounded} is satisfied if 
$$
\frac{d}{ds} \int^\infty_{\frac{1}{s\sqrt{\log \frac{1}{s}}}} \frac{1}{\Psi(v)}\,dv =\frac{1}{\Psi\left(\frac{1}{s\sqrt{\log \frac{1}{s}}}\right)}\left(\sqrt{\log \frac{1}{s}} - \frac{1}{2\sqrt{\log \frac{1}{s}}}\right)\frac{1}{s^2 \log \frac{1}{s}}
$$
stays bounded in a right neighborhood of $s=0$. This is true if
\begin{equation}
    \label{cleaner reformulation}
    \frac{1}{s^2\sqrt{\log \frac{1}{s}}}\lesssim \Psi\left(\frac{1}{s\sqrt{\log \frac{1}{s}}}\right)\qquad \forall\, 0<s\ll 1.
\end{equation}
By Proposition \ref{P:two nash ineq} (2) we can choose $\Psi=\left(\gamma^{-1}\right)^2$ for an increasing function $\gamma$ such that $\gamma(z) \sim \frac{z}{\sqrt[4]{\log z}}$ for $z\gg 1$. It is immediate to check that \eqref{cleaner reformulation} is satisfied with such a choice. This concludes the proof of \eqref{reformulating the bound for delort}, and thus of \eqref{enstroph bound for delort} as well.

 We conclude, for $\nu T<1$ and $T>\delta>0$, that 
 \begin{align}
     \nu \int_\delta^T \|\omega^\nu(\tau)\|^2_{L^2}\,d\tau &\lesssim \int_\delta^T \frac{1}{\tau \sqrt{|\log (\nu \tau)|}}\,d\tau \\
     &= \int_{\log \frac{1}{\nu T}}^{\log \frac{1}{\nu \delta}} \frac{1}{\sqrt{\tau}}\,d\tau \\
     &=2\left(\sqrt{{\log \frac{1}{\nu \delta}}} - \sqrt{{\log \frac{1}{\nu T}}}\right) \\
     &=2\frac{\log \frac{1}{\nu \delta} - \log \frac{1}{\nu T}}{\sqrt{{\log \frac{1}{\nu \delta}}} + \sqrt{{\log \frac{1}{\nu T}}}} \\
     &\lesssim \frac{\log \frac{T}{\delta}}{\sqrt{|\log (\nu T)|}}, 
 \end{align}
 with an implicit constant that does not depend on any of the parameters $\nu,\delta,T$.
\end{proof}

Corollary  \ref{C:time scale} follows as an almost direct consequence of Theorem \ref{T:main}.

\begin{proof}[Proof of Corollary  \ref{C:time scale}]
The claim $(a')$ is obvious from Theorem \ref{T:main} $(a)$. We are left to prove $(b')$. Let $\eps>0$. Thanks to the strong compactness of the initial velocities, by \cite{DRP25}*{Proposition 3.1} we find  $\delta>0$ such that 
\begin{equation}
    \label{no dissipation for short times}
    \limsup_{\nu\rightarrow 0} \nu\int_0^\delta \|\omega^\nu(\tau)\|^2_{L^2}\,d\tau<\eps.
\end{equation}
Let $\kappa \in (0,\frac12)$ be fixed and $\{T_\nu\}_\nu$ a sequence of positive numbers such that 
\begin{equation}
    \label{time scale restriction for delort}
    \lim_{\nu\rightarrow 0} \frac{T_\nu}{e^{|\log \nu|^\kappa}}=0.
\end{equation}
In particular, it must be $\nu T_\nu<1$ for all $\nu>0$ sufficiently small. We can thus apply Theorem \ref{T:main} $(b)$ to get 
$$
 \nu \int_\delta^{T_\nu} \|\omega^\nu(\tau)\|^2_{L^2}\,d\tau\lesssim \frac{\log \frac{T_\nu}{\delta}}{\sqrt{|\log \left(\nu T_\nu\right)|}}.
$$
If  $\{T_\nu\}_\nu$ is bounded, we can trivially conclude. Thus assume $T_\nu\rightarrow \infty$ as $\nu\rightarrow 0$. In this case, for all sufficiently small\footnote{Depending on $\delta>0$ as well.} $\nu>0$, we have 
\begin{equation}\label{ugly bound}
 \nu \int_\delta^{T_\nu} \|\omega^\nu(\tau)\|^2_{L^2}\,d\tau\lesssim \frac{\log T_\nu}{\sqrt{|\log (\nu T_\nu)|}}=\frac{1}{\sqrt{\left| \frac{\log \nu}{(\log T_\nu)^2}+\frac{1}{\log T_\nu}\right|}}.
\end{equation}
Note that $\left(\log T_\nu\right)^{-1}\rightarrow 0$ as $\nu\rightarrow 0$. Moreover, the condition  \eqref{time scale restriction for delort} implies $e^{|\log \nu|^\kappa}>T_\nu$ for all $\nu\ll 1$. Equivalently, it must hold $|\log \nu|^{2\kappa}>(\log T_\nu)^2$, from which 
$$
\frac{|\log \nu|}{(\log T_\nu)^2}>|\log \nu|^{1-2\kappa}\rightarrow \infty\qquad \text{as } \nu\rightarrow 0,
$$
since $\kappa <\frac12$. These considerations prove that the right-hand-side in \eqref{ugly bound} vanishes, that is
$$
  \limsup_{\nu\rightarrow 0} \nu \int_\delta^{T_\nu} \|\omega^\nu(\tau)\|^2_{L^2}\,d\tau =0.
$$
Together with \eqref{no dissipation for short times}, this yields to 
$$
 \limsup_{\nu\rightarrow 0} \nu \int_0^{T_\nu} \|\omega^\nu(\tau)\|^2_{L^2}\,d\tau<\eps.
$$
The proof is concluded since $\eps>0$ was arbitrary.
\end{proof}

\begin{remark}
    \label{R:infinite energy}
    As it should be clear from the computation above, all the bounds depend only on $\sup_{\nu>0}\|\omega^\nu_0\|_{\mathcal M}$ and $\mathbb{M}_\omega$. They are, in particular, independent of $\|u^\nu_0\|_{L^2}$, which is then allowed to grow indefinitely as $\nu\rightarrow 0$. The only reason why we are still requiring $\{u^\nu_0\}_\nu\subset L^2(\T^2)$ in Theorem \ref{T:main} and Corollary  \ref{C:time scale} is to make sure that the Navier--Stokes equations are well-posed. On the whole space $\R^2$, the well-posedness is known for any measure initial vorticity \cite{GG05}, thus allowing initial data with infinite kinetic energy. We expect that, with suitable modifications, the methods from \cite{GG05} should extend to the spatially periodic setting as well, allowing to remove the assumption $\{u^\nu_0\}_\nu\subset L^2(\T^2)$ from Theorem \ref{T:main} and Corollary  \ref{C:time scale}. In fact, being or not in a well-posedness setting is not really essential to our purposes as the arguments apply to any, a priori given, solution to \eqref{NS-Vort} as soon as $\sup_{\nu>0}\|\omega^\nu_0\|_{\mathcal M}<\infty$.
\end{remark}

\begin{remark}
    \label{R:general AD eq}
    In none of the above proofs we have used the specific relation between $\omega^\nu$ and $u^\nu$. In particular, all the results still hold for any scalar, active or not, solving
    $$
    \partial_t\theta^\nu +u^\nu\cdot \nabla \theta^\nu =\nu \Delta \theta^\nu,
    $$
    as soon as $\div u^\nu =0$, $\sup_{\nu>0}\|\theta^\nu_0\|_{\mathcal M}<\infty$ and $\mathbb{M}_\theta(r)\rightarrow 0$ as $r\rightarrow 0$.
\end{remark}

We conclude this section by proving Proposition \ref{P:saturate algebraic rate}, that is by constructing an explicit example that saturates the rate proved in Theorem \ref{T:main} (a) for $\alpha=1$.

\begin{proof}[Proof of Proposition \ref{P:saturate algebraic rate}]
Let 
$$
v_0(x):= \frac{x^\perp}{|x|^2} \mathbbm{1}_{B^c_1(0)}(x)\qquad\forall x\in \R^2.
$$
Direct computations show $\div v_0 =0$ and $\curl v_0 = \mathcal H^1\llcorner \mathbb{S}^1$ in $\mathcal D'(\R^2)$. Let $\chi \in C^\infty_c([0,2))$ be a cutoff function, constantly equal to $1$ on $[0,1]$. Set $\tilde \chi (x):=\chi (|x|)$ for all $x\in \R^2$. Consequently, we define $u_0:=\tilde \chi v_0$. Since 
$$
v_0(x)\cdot \nabla \tilde \chi (x) = v_0(x)\cdot \frac{x}{|x|}\chi'(|x|) = 0 \qquad \forall x\in\R^2,
$$
by $\div v_0=0$ we deduce $\div u_0=0$ as well. Moreover, $u_0$ is bounded and compactly supported in $B_2(0)$. Its $\curl$ is given by 
$$
\omega_0 :=\curl u_0 =-\div u_0^\perp = \frac{\chi'(|x|)}{|x|} + \mathcal H^1\llcorner \mathbb{S}^1\qquad \text{in } \mathcal D'(\R^2).
$$
For any $\nu>0$ solve 
$$
  \begin{cases}
\partial_t \omega^\nu =\nu\Delta \omega^\nu \\[4pt]
\omega^\nu(\cdot, 0)=\omega_0
\end{cases}
$$
on $\R^2\times [0,\infty)$. Since $\omega_0$ is radially symmetric,  $\{\omega^\nu\}_\nu$ solve \eqref{NS-Vort} as well.

\underline{\textsc{Proof of \eqref{1d rate saturation}}}. Let 
$$
\omega_{0,1}(x) := \frac{\chi'(|x|)}{|x|}\qquad \text{and}\qquad \omega_{0,2} = \mathcal H^1\llcorner \mathbb{S}^1.
$$
By linearity it must be $\omega^\nu = \omega^\nu_1+\omega^\nu_2$, where 
\begin{equation}\label{linear split}
 \begin{cases}
\partial_t \omega_1^\nu =\nu\Delta \omega_1^\nu \\[4pt]
\omega_1^\nu(\cdot, 0)=\omega_{0,1}
\end{cases}
\qquad \text{and}\qquad  \begin{cases}
\partial_t \omega_2^\nu =\nu\Delta \omega_2^\nu \\[4pt]
\omega_2^\nu(\cdot, 0)=\omega_{0,2}.
\end{cases}
\end{equation}
Since $\omega_{0,1}\in L^2(\R^2)$, we have 
\begin{equation}
    \label{bound on smooth part}
    \|\omega^\nu_1(t)\|_{L^2}\leq  \|\omega_{0,1}\|_{L^2}\qquad \forall \nu,t>0.
\end{equation}
We claim that 
\begin{equation}
    \label{claim bound on singular part}
    \nu\int_0^1 \|\omega^\nu_2(\tau)\|^2_{L^2}\,d\tau\gtrsim \sqrt\nu\qquad \forall \nu\ll 1.
\end{equation}
Given the claim, the lower bound \eqref{1d rate saturation}  follows. Indeed, by the Young's inequality\footnote{More precisely, we are using $2ab\geq -4a^2-\frac{b^2}{4}$ for any  $a,b\in\R$.} and \eqref{bound on smooth part} we deduce 
\begin{align}
\nu\int_0^1 \|\omega^\nu(\tau)\|^2_{L^2}\,d\tau&= \nu\int_0^1 \|\omega^\nu_1(\tau)\|^2_{L^2}\,d\tau + \nu\int_0^1 \|\omega^\nu_2(\tau)\|^2_{L^2}\,d\tau  \\
&\quad +2\nu \int_0^1 \big(\omega^\nu_1(\tau),\omega^\nu_2(\tau)\big)\,d\tau \\
&\geq \frac{3}{4}\nu \int_0^1 \|\omega^\nu_2(\tau)\|^2_{L^2}\,d\tau - 3\nu \int_0^1 \|\omega^\nu_1(\tau)\|^2_{L^2}\,d\tau \\
&\gtrsim \sqrt{\nu}-\nu, 
\end{align}
from which we get \eqref{1d rate saturation}. 

We want to prove \eqref{claim bound on singular part}. In the Fourier variable, the solution $\omega_2^\nu$ takes the explicit form 
$$
\hat \omega_2^\nu(\xi,t)=\hat \omega_{0,2}(\xi) e^{-\nu t|\xi|^2}\qquad \forall \xi\in \R^2,\, t>0.
$$
Then, by the Plancherel identity 
$$
\|\omega_2^\nu(t)\|^2_{L^2} = \int_{\R^2}|\hat \omega_{0,2}(\xi)|^2 e^{-2\nu t|\xi|^2}\,d\xi.
$$
We will prove the stronger lower bound 
\begin{equation}
    \label{stronger lower bound}
    \int_{\R^2}|\hat \omega_{0,2}(\xi)|^2 e^{-2\nu t|\xi|^2}\,d\xi\gtrsim 1+\frac{1}{\sqrt{\nu t}}\qquad \text{if } \nu t<1,
\end{equation}
which immediately implies \eqref{claim bound on singular part} after integration in time. Since $\omega_{0,2}$ is a finite measure, we have $
\hat \omega_{0,2}\in L^\infty(\R^2)$. Moreover, $\hat \omega_{0,2}$ is radial and can be written as $\hat \omega_{0,2}(\xi) = 2\pi J_0(|\xi|)$ where $J_0$ is the Bessel function of the first kind of order $0$. It is well-known that 
$$
J_0(|\xi|)\sim\frac{\cos \left(|\xi|-\frac{\pi}{4}\right)}{\sqrt{|\xi|}}\qquad \forall |\xi|>1.
$$
Also, since $\nu t<1$, we have 
$$
\int_{B_1(0)}|\hat \omega_{0,2}(\xi)|^2 e^{-2\nu t|\xi|^2}\,d\xi\geq \frac{1}{e^2}\int_{B_1(0)}|\hat \omega_{0,2}(\xi)|^2 \,d\xi.
$$
Therefore 
\begin{align}
     \int_{\R^2}|\hat \omega_{0,2}(\xi)|^2 e^{-2\nu t|\xi|^2}\,d\xi&\geq  \frac{1}{e^2} \int_{B_1(0)}|\hat \omega_{0,2}(\xi)|^2 \,d\xi+  \int_{B_1^c(0)}|\hat \omega_{0,2}(\xi)|^2 e^{-2\nu t|\xi|^2}\,d\xi \\
     &\gtrsim 1 +  \int_{B_1^c(0)}\frac{\cos^2\left(|\xi|-\frac{\pi}{4}\right)}{|\xi|} e^{-2\nu t|\xi|^2}\,d\xi \\
     &\gtrsim 1+\frac{1}{\sqrt{\nu t}} \int_{\sqrt{\nu t}}^\infty \cos^2\left(\frac{r}{\sqrt{\nu t}}-\frac{\pi}{4}\right) e^{-2r^2}\,dr.\label{bound for RL}
\end{align}
By the formula $2\cos^2 r = 1+\cos (2r)$ we can rewrite 
$$
\cos^2\left(\frac{r}{\sqrt{\nu t}}-\frac{\pi}{4}\right)=\frac{1+\cos \left(\frac{2r}{\sqrt{\nu t}}-\frac{\pi}{2}\right)}{2}= \frac{1+\sin \left(\frac{2r}{\sqrt{\nu t}}\right)}{2}.
$$
Since by the Riemann--Lebesgue lemma
$$
\lim_{n\rightarrow \infty} \int_{\frac{1}{n}}^\infty \cos^2 \left(nr-\frac{\pi}{4}\right) e^{-2r^2}\,dr = \frac{1}{2} \int_0^\infty e^{-2r^2}\,dr,
$$
by using it into \eqref{bound for RL} we conclude \eqref{stronger lower bound}.

\underline{\textsc{Proof of \eqref{1d decay on balls}}}. Let $x\in \R^2$ be arbitrarily fixed. Recall the splitting $\omega^\nu = \omega^\nu_1+\omega^\nu_2$ from \eqref{linear split}. Note that $\omega_2^\nu$ stays positive in the evolution. Therefore, if $\varphi_r\in C^\infty_c(B_{4r}(x))$ is a nonnegative cutoff such that 
$$
\varphi\big|_{B_r(x)}\equiv 1 \qquad \text{and}\qquad  |\nabla \varphi_r|\leq r^{-1},
$$
by  \eqref{bound on smooth part} we deduce
\begin{align}
\int_{B_r(x)}|\omega^\nu(y,t)|\,dy&\leq  \int_{B_r(x)}|\omega_1^\nu(y,t)|\,dy + \int_{B_r(x)}|\omega^\nu_2(y,t)|\,dy \\
&\lesssim r + \int_{B_{4r}(x)} \omega^\nu_2(y,t)\varphi_r(y)\,dy.\label{almost done}
\end{align}
Since $\omega_{0,2} = \curl v_0$, it must be $\omega^\nu_2 = \curl v^\nu$ with 
$$
\begin{cases}
\partial_t v^\nu =\nu\Delta v^\nu \\[4pt]
v^\nu(\cdot, 0)=v_{0}.
\end{cases}
$$
Moreover, by the maximum principle $\|v^\nu(t)\|_{L^\infty}\leq \|v_0\|_{L^\infty}$ for all $\nu, t>0$. Therefore, we can integrate by parts the last term in the right-hand-side of \eqref{almost done} to conclude 
\begin{align}
   \int_{B_r(x)}|\omega^\nu(y,t)|\,dy&\lesssim r - \int_{B_{4r}(x)} v^\nu(y,t)\cdot \nabla^\perp \varphi_r(y)\,dy \\
   &\lesssim r + r^2 \|\nabla \varphi_r\|_{L^\infty}  \\
   &\lesssim r, 
\end{align}
for an implicit constant that does not depend on $x,r,\nu$ and $t$.
\end{proof}

\begin{remark}
    In the previous proof, the sharpness of the vanishing rate of the dissipation has been obtained as a consequence of the fact that  the Fourier transform of the one-dimensional Lebesgue measure on $\mathbb{S}^1$, which we have denoted by $\hat \omega_{0,2}$, decays, modulo oscillations, exactly like $|\hat \omega_{0,2}(\xi)|^2\sim |\xi|^{-1}$ for all $|\xi|>1$.
    At least for any $\alpha\in (0,\frac12)$, the existence of measures $\mu$ such that $\sup_{x}\mu(B_r(x))\lesssim r^\alpha$ and whose quadratic Fourier spherical average  decays exactly like
    $$
    \int_{\mathbb{S}^1} |\hat \mu(R\xi)|^2 \, d\mathcal{H}^1(\xi)\sim |R|^{-\alpha}\qquad \forall R>1
    $$
  is due to Mattila \cite{Mattila87}. See also \cites{LR19,Wolff99} and references therein for the case $\alpha>\frac12$. By essentially repeating the same computations we have done in the proof of Proposition \ref{P:saturate algebraic rate}, these can be used to prove that the heat equation, with such a measure as an initial datum, saturates the corresponding vanishing rate of the dissipation. Being outside of the radial symmetry, this does not seem to trivially translate into a statement for the Navier--Stokes equations.
\end{remark}

\section{Failed sharp attempts in the Delort's class}\label{S:failed tries}
A natural question is whether the results obtained in this work, more specifically the rates in Theorem \ref{T:main}, are sharp. As shown in the previous section, Proposition \ref{P:saturate algebraic rate} provides an affirmative answer for Theorem \ref{T:main} (a) when $\alpha=1$ by constructing an explicit example saturating the rate. We expect the result to extend to any $\alpha \in (0,2)$, thus covering the entire case (a). For Theorem \ref{T:main} (b), instead, sharpness can only be conjectured at this stage. Our main interest in establishing the sharpness of Theorem \ref{T:main} (b) is towards Conjecture \ref{Conj}. In this section we present a couple of attempts which turned out to be unsuccessful. Nonetheless, we believe it is worth discussing them. 

 The original goal of these attempts was to find a sequence of initial velocities $\{u_0^{\nu}\}_{\nu}$, possibly bounded\footnote{Although that is not required in Theorem \ref{T:main}, it is desirable to keep the kinetic energy bounded if aiming at Conjecture \ref{Conj}.} in $L^2(\mathbb{T}^2)$, such that the sequence of initial vorticities $\{\omega_0^{\nu}\}_{\nu} \subset \mathcal{M}(\mathbb{T}^2)$ is bounded and the corresponding solutions $\{\omega^{\nu}\}_{\nu}$ to the Navier--Stokes equations \eqref{NS-Vort} satisfy 
 \begin{equation}
     \label{attempt M log decay}
     \mathbb{M}_\omega(r) \lesssim \frac{1}{\sqrt{|\log r|}} \qquad \forall r\in \left(0,\frac12\right)
 \end{equation}
 and
  \begin{equation}
     \label{attempt dissipation sharp}
   \nu \int_\delta^1 \|\omega^\nu (\tau)\|^2_{L^2}\,d\tau \gtrsim \frac{1}{\sqrt{|\log \nu|}} \qquad \forall 0 < \nu \ll 1,
 \end{equation}
 for some $\delta\in (0,1)$. Note that, in view of the proof of the improved Nash inequality given in \cref{S:nash ineq}, to hope for \eqref{attempt dissipation sharp} it is necessary that the logarithmic decay \eqref{attempt M log decay} is optimal. Although it is quite easy to produce solutions saturating \eqref{attempt M log decay}, it turns out that such examples have a dissipation $\lesssim |\log \nu|^{-1}$, ruling out the possibility to achieve \eqref{attempt dissipation sharp}. In \cref{S:conclusions} we will discuss why we believe this happens.

 We  restrict to consider radially symmetric vorticities since, in this case, $u^\nu\cdot \nabla \omega^\nu\equiv 0$ and  \eqref{NS-Vort} reduces to the heat equation which is much simpler to handle. For simplicity, we work on the whole space $\mathbb{R}^2$. Therefore, in the definition of the absolute vorticity on balls $\mathbb{M}_{\omega}(r)$ the supremum is taken over all $x \in \mathbb{R}^2$. The first example is with a sequence of initial data obtained by rescaling a smooth and compactly supported vorticity profile, while, in the second case, we consider a fixed integrable initial vorticity with the sharp logarithmic decay on balls.

\subsection{Construction by rescaling}\label{S:first ex} The first attempt is based on a spatial rescaling inspired by \cite{majda1993remarks}. In particular, the example we construct in this section proves Proposition \ref{P:fake delort sharp}. 

Let $\mu_0\in C^\infty_c(B_1(0))$ be  nontrivial, nonnegative and radially symmetric. The vector field
\begin{equation}
    \label{vort veloc radial formula}
v_0(x):=\frac{x^\perp}{|x|^2}\int_0^{|x|} r \mu_0(r)\,dr
\end{equation}
satisfies $\div v_0=0$ and $\curl v_0 = \mu_0$ in $\mathcal D'(\R^2)$. Note that, although $v_0\in C^\infty(\R^2)$, $v_0\not \in L^2(\R^2)$ since $|v_0(x)|\sim |x|^{-1}$ for $|x|>1$. Solve 
$$
\left\{\begin{array}{ll}
\partial_t \mu = \Delta \mu \\
\mu(\cdot, 0)=\mu_0
\end{array}\right.
$$
on $\R^2\times [0,\infty)$. Consequently, for any $\nu>0$, we set 
\begin{equation}
    \label{rescaling heat}
    \mu^\nu(x,t):=\frac{1}{\nu\sqrt{|\log \nu|}}\mu\left(\frac{x}{\sqrt{\nu}},t\right)\qquad \forall x\in \R^2,\, t\geq 0.
\end{equation}
This solves 
$$
\left\{\begin{array}{ll}
\partial_t \mu^\nu = \nu \Delta \mu^\nu \\
\mu^\nu(\cdot, 0)=\mu_0^\nu,
\end{array}\right.
$$
with 
$$
 \mu_0^\nu(x):=\frac{1}{\nu\sqrt{|\log \nu|}}\mu_0\left(\frac{x}{\sqrt{\nu}}\right)\qquad \forall x\in \R^2.
$$
The corresponding initial velocity is given by 
$$
v_0^\nu(x):= \frac{1}{\sqrt{\nu |\log \nu|}} v_0\left(\frac{x}{\sqrt{\nu}}\right)\qquad \forall x\in \R^2.
$$
Moreover
\begin{equation}
    \label{L2 loc}
    \sup_{\nu>0}\int_{B_2(0)} |v_0^\nu(x)|^2\,dx<\infty. 
\end{equation}
Note, however, that $v_0^\nu\not\in L^2(\R^2)$ since $v_0\not\in L^2(\R^2)$. We will thus cut it off at infinity. To do that, let $\chi \in C^\infty_c([0,2))$ be a nonnegative function, constantly equal to $1$ on $[0,1]$. Set $\tilde \chi (x):=\chi (|x|)$ for all $x\in \R^2$. Consequently, we define $u^\nu_0:=\tilde \chi v^\nu_0$.
Note that 
$$
\left(\div u^\nu_0\right)(x)=u^\nu_0(x)\cdot \tilde \chi (x) = u^\nu_0(x)\cdot \frac{x}{|x|} \chi'(|x|) = 0 \qquad \forall x\in \R^2.
$$
Moreover, by \eqref{L2 loc} we deduce that $\{u^\nu_0\}_\nu\subset L^2(\R^2)$ is bounded. Its $\curl$ is given by 
$$
\omega_0^\nu (x) =  \chi ' (|x|) \frac{x^\perp}{|x|}\cdot v^\nu_0(x) +\mu_0^\nu(x)\qquad \forall x\in \R^2.
$$
Let $\{\omega^\nu\}_\nu$ be the corresponding sequence of solutions to \eqref{NS-Vort}. By the radial symmetry of the initial data, these solve the heat equation. Therefore, denoting by 
$$
f^\nu_0(x):=\chi ' (|x|) \frac{x^\perp}{|x|}\cdot v^\nu_0(x),
$$
by linearity it must be $\omega^\nu = f^\nu +\mu^\nu$, where $f^\nu$ solves the heat equation with initial datum $f^\nu_0$ and $\mu^\nu$ is given by \eqref{rescaling heat}. From \eqref{L2 loc} we deduce 
\begin{equation}
    \label{bound for the L2 part}
    \|f^\nu(t)\|_{L^2}\leq \|f^\nu_0\|_{L^2}\lesssim 1\qquad \forall \nu,  t>0.
\end{equation}
Moreover 
\begin{equation}\label{L2 norm of rescaling}
\|\mu^\nu(t)\|^2_{L^2} = \frac{1}{\nu |\log \nu|}\|\mu(t)\|^2_{L^2}.
\end{equation}
By putting \eqref{bound for the L2 part} and \eqref{L2 norm of rescaling} together, we obtain 
\begin{align}
    \|\omega^\nu(t)\|^2_{L^2} &= \|f^\nu(t)\|^2_{L^2} + \|\mu^\nu(t)\|^2_{L^2} + 2\big(f^\nu(t), \mu^\nu(t)\big) \\
    &\geq \frac{3}{4} \|\mu^\nu(t)\|^2_{L^2} - 3 \|f^\nu(t)\|^2_{L^2} \\
    &\gtrsim \frac{1}{\nu |\log \nu|} -1, 
\end{align}
where to obtain the second last inequality we have used the Young's inequality $2ab\geq -4a^2 - \frac{b^2}{4}$ to bound the inner product from below. Also 
$$
\|\omega^\nu(t)\|^2_{L^2}\leq 2\left(\|f^\nu(t)\|^2_{L^2}+ \|\mu^\nu(t)\|^2_{L^2}\right)\lesssim 1+ \frac{1}{\nu |\log \nu|}
$$
for all $\nu >0$ and all $t\in (0,1)$. We have thus proved Proposition \ref{P:fake delort sharp}.

Let us summarize the above construction. We will ignore the ``absolutely continuous'' part $f^\nu$ since it only contributes with a lower order term\footnote{The reason why $f^\nu$ arose was only to cutoff  the velocity at infinity in order to make with finite kinetic energy, globally on $\R^2$.}. The main part $\mu^\nu$ has been obtained by rescaling the nonnegative function $\mu$, so that it concentrates at the origin, for all times $t\geq 0$.  Since point concentrations in the vorticity would not be consistent with a velocity having locally finite kinetic energy\footnote{See for instance the explicit formula \eqref{vort veloc radial formula} in the case of a radial vorticity.}, the additional term $|\log \nu|$ in the rescaling \eqref{rescaling heat} is necessary. Here by ``additional'' we mean with respect to the one that keeps the $L^1(\R^2)$ norm invariant. In addition, this also ensures that 
\begin{equation}
    \sup_{\substack{x \in \R^2  \\ \nu,t>0}} \int_{B_r(x)}|\mu^\nu(y,t)|\,dy = \sup_{\nu>0} \int_{B_r (0)} |\mu^\nu_0(y)|\,dy \sim \frac{1}{\sqrt{|\log r|}}\qquad \forall r\in \left(0,\frac12\right). 
\end{equation}
However, although the decay on balls is sharp, the presence of $|\log \nu|$ in \eqref{rescaling heat} gives 
$$
\|\mu^\nu_0\|_{L^1}= \frac{1}{\sqrt{|\log \nu|}}\|\mu_0\|_{L^1}\qquad \forall \nu>0.
$$
Therefore, not only the vorticity gives small mass to small balls, but it also vanishes strongly in $L^1(\R^2)$. 

The fact that all examples of this kind cannot give the conjecturally sharp rate can be seen with very little effort. Indeed, by applying the classical Nash inequality \eqref{classical nash} to
$$
\frac{d}{dt}\|\omega^{\nu}(t)\|_{L^2}^2 = -2 \nu \|\nabla \omega^{\nu}\|_{L^2}^2
$$
we obtain
$$
\frac{d}{dt}\|\omega^{\nu}(t)\|_{L^2}^2\lesssim -\frac{\nu}{\|\omega^\nu_0\|^2_{L^1}}  \|\omega^{\nu}(t)\|_{L^2}^4.
$$
By integrating the differential inequality, we deduce
$$
\|\omega^\nu(t)\|^2_{L^2}\lesssim \frac{\|\omega^\nu_0\|^2_{L^1}}{\nu t}\qquad \forall \nu, t>0, 
$$
from which 
\begin{equation}\label{general bound with L1 initial mass}
\nu \int_{\delta}^1 \|\omega^{\nu}(\tau)\|_{L^2}^2\,d\tau \lesssim |\log\delta| \|\omega^\nu_0\|^2_{L^1}\qquad \forall \nu,\delta>0.
\end{equation}
It follows that, whenever the total mass of the initial vorticities vanishes as $|\log \nu|^{-\sfrac12}$, the dissipation is bounded by $|\log \nu|^{-1}$, a square away from what we believe it should be the sharp one. 

It is also clear that, when the sequence vanishes in $L^1(\R^2)$, there is an additional small term in the proof of the improved Nash inequality, making \eqref{improved Nash from ELL} suboptimal.

\subsection{Construction with an integrable initial datum}\label{S:second ex} The approach described before is based on a specific variable rescaling. This  allows to built a sequence of solutions to \eqref{NS-Vort} having the desired decay on balls \eqref{attempt M log decay}, but whose total mass on the whole space vanishes, thus preventing the possibility to satisfy \eqref{attempt dissipation sharp}. This leads us to pursue a different strategy and, in particular, to search for radial solutions to the heat equation satisfying \eqref{attempt M log decay} with a fixed initial datum, so that its total mass is of order $1$. Since it is not necessary to consider singular measures to saturate \eqref{attempt M log decay}, we will work with an integrable function. To this aim, we introduce the function $\omega_0$ defined as
\begin{equation}\label{def of L1 vort}
\omega_0(x):=\frac{1}{|x|^2 |\log |x||^{\frac{3}{2}}}\mathbbm{1}_{B_{\frac{1}{2}}(0)}(x)\qquad \forall x\in \R^2.
\end{equation}
We then solve, for any $\nu>0$, the following Cauchy problem \begin{equation}\label{heat equation 3} 
\left\{\begin{array}{ll}
\partial_t \omega^{\nu} = \nu \Delta \omega^{\nu} \\
\omega^{\nu}(\cdot, 0)=\omega_0
\end{array}\right.\end{equation}
on $\R^2\times [0,\infty)$. By the radial symmetry of the initial datum, $\{\omega^\nu\}_\nu$ solves \eqref{NS-Vort}. Notice that, in contrast to the attempt presented before, here the initial datum does not depend on $\nu$. Moreover 
\begin{equation}\label{L1 funct with log decay}
       \sup_{x\in \R^2} \int_{B_r(x)}|\omega_0(y)|\,dy = \int_{B_r(0)}|\omega_0(y)|\,dy  \sim \frac{1}{\sqrt{|\log r|}}\qquad \forall r\in \left(0,\frac12\right).
    \end{equation}
    Such property can be easily seen to be propagated in time by \eqref{heat equation 3}, proving \eqref{attempt M log decay}.
        
        The initial velocity $u_0$ corresponding to $\omega_0$ borderline fails to have finite kinetic energy in the sense specified below. Due to \eqref{L1 funct with log decay}, for small values of $|x|$ it holds $$u_0(x)=\frac{x^\perp}{|x|^2}\int_0^{|x|}r\omega_0(r)\,dr =\frac{x^\perp}{|x|^2}\omega_0(B_{|x|}(0))\sim \frac{x^\perp}{|x|^2 \sqrt{|\log|x||}}.$$
       Therefore $u_0\not \in L^2(B_R(0))$ for any $R>0$. However, if we would have increased the power of the logarithm in \eqref{def of L1 vort} by an $\eps>0$, the corresponding initial velocity would be in $L^2_{\text{\rm loc}}(\R^2)$. Then, as already done in \cref{S:first ex}, a further multiplication with a smooth compactly supported radial function would make it with finite kinetic energy globally on $\R^2$. Since these little modifications would not add any substantial difference\footnote{Adding an $\eps>0$ to the power of the logarithm in \eqref{def of L1 vort} would in fact provide an even faster rate of the dissipation. Cutting off the velocity at infinity only adds a lower order term, which does not modify the rate.}, we will stick to \eqref{def of L1 vort} to lighten the calculations.
       
       We want to estimate the energy dissipation term$$\nu \int_{\delta}^1 \|\omega^{\nu}(\tau)\|^2_{L^2}\,d\tau$$
       for some $\delta\in (0,1)$. We will see that the rate is again $\lesssim |\log\nu|^{-1}$, thus failing to match \eqref{attempt dissipation sharp}. We begin by providing an estimate for $\|\omega^{\nu}(t)\|_{L^2}^2$ for all $t>0$. We pass to Fourier variables in \eqref{heat equation 3} to get $$\hat{\omega}^{\nu}(\xi,t)=\hat{\omega}_0(\xi)e^{-\nu t|\xi|^2}.$$ We  want to derive an upper bound for the Fourier transform of the initial datum  $$\hat{\omega}_0(\xi)=\int_{B_{\frac{1}{2}}(0)}\frac{1}{|x|^2|\log|x||^{\frac{3}{2}}}e^{-ix\cdot \xi}\,dx \qquad \forall \xi\in \R^2.$$
       By the coarea formula we have $$\hat{\omega}_0(\xi)=\int_0^{\frac{1}{2}}\frac{1}{r|\log r|^{\frac{3}{2}}}S(r\xi)\,dr\qquad \text{with}\qquad   S(r\xi):=\int_{\mathbb{S}^1}e^{-iy\cdot \xi r}\,d\mathcal{H}^1(y).$$Recalling the computations carried out in the proof of Proposition \ref{P:saturate algebraic rate} from \cref{S:enstroph bounds}, we know that 
    $$
    |S(r\xi)| \lesssim \frac{1}{\sqrt{r|\xi|}} \qquad \text{if } r|\xi|>1.
    $$
On the other hand, for $r |\xi|\leq 1$, the bound $|S(r\xi)| \leq 2\pi$ is trivial. Consequently, by splitting the integral as $$\hat{\omega}_0(\xi)=\int_0^{\frac{1}{|\xi|}}\frac{1}{r|\log r|^{\frac{3}{2}}}S(r\xi)\,dr+\int_{\frac{1}{|\xi|}}^{\frac{1}{2}}\frac{1}{r|\log r|^{\frac{3}{2}}}S(r\xi)\,dr \qquad \forall |\xi|>4,$$ 
we obtain that
\begin{align}
|\hat{\omega}_0(\xi)| &\lesssim \int_0^{\frac{1}{|\xi|}}\frac{1}{r |\log r|^{\frac{3}{2}}}\,dr+\frac{1}{\sqrt{|\xi|}}\int_{\frac{1}{|\xi|}}^{\frac{1}{2}}\frac{1}{r^{\frac{3}{2}}|\log r|^{\frac{3}{2}}}\,dr  \\
&\lesssim \frac{1}{\sqrt{|\log |\xi||}}+\frac{1}{\sqrt{|\xi|}}\frac{\sqrt{|\xi|}}{|\log |\xi||^\frac32}\int_{\frac{1}{|\xi|}}^{\frac{1}{2}}\frac{1}{r}\,dr \\
&\lesssim \frac{1}{\sqrt{|\log |\xi||}}  
\end{align}
for all $|\xi|>4$. Moreover
$$
|\hat \omega_0(\xi)|\leq \|\omega_0\|_{L^1}\qquad \forall \xi\in \R^2.
$$
To summarize, we have 
$$
|\hat \omega_0(\xi)|\lesssim \begin{cases}
1 & \text{if }\; |\xi|\leq 4\\[4pt]
\frac{1}{\sqrt{\log |\xi|}} & \text{if }\; |\xi|>4.
\end{cases}
$$
By Plancherel, we then deduce 
\begin{align}
    \|\omega^\nu (t)\|^2_{L^2} &=  \|\hat \omega^\nu (t)\|^2_{L^2} \\
    & =\int_{B_4(0)} |\hat \omega^\nu (\xi,t)|^2\,d\xi +  \int_{B^c_4(0)} |\hat \omega^\nu (\xi,t)|^2\,d\xi \\
    &\lesssim \int_{B_4(0)} e^{-2\nu t|\xi|^2}\,d\xi + \int_{B^c_4(0)} \frac{e^{-2\nu t|\xi|^2}}{\log |\xi|}\,d\xi.\label{split in B4 and complement}
\end{align}
The first term enjoys the trivial estimate
\begin{equation}
    \label{trivial estimate for first}
    \int_{B_4(0)} e^{-2\nu t|\xi|^2}\,d\xi\lesssim 1\qquad \forall \nu,t>0.
\end{equation}
The second one requires some extra care. By integrating it in polar coordinates and changing variables, we bound it as 
\begin{equation}
    \label{bound in polar coord}
     \int_{B^c_4(0)} \frac{e^{-2\nu t|\xi|^2}}{\log |\xi|}\,d\xi\leq \int_4^\infty \frac{r e^{-\nu tr^2}}{\log r}\,dr  = \frac{1}{\nu t |\log (\nu t)|} \int_{4\sqrt{\nu t}}^\infty \frac{r e^{-r^2}}{\frac12  - \frac{\log r}{\log (\nu t)} }\,dr.
\end{equation}
Therefore, by using \eqref{trivial estimate for first} and \eqref{bound in polar coord} into \eqref{split in B4 and complement} we obtain 
\begin{equation}
    \label{almost last bound}
    \|\omega^\nu (t)\|^2_{L^2} \lesssim 1+\frac{1}{\nu t |\log (\nu t)|} \int_{4\sqrt{\nu t}}^\infty \frac{r e^{-r^2}}{\frac12  - \frac{\log r}{\log (\nu t)} }\,dr.
\end{equation}
We now claim that 
\begin{equation}
    \label{claim integral bounded}
    \sup_{\eps\in (0,1)}\int_{4\sqrt{\eps}}^\infty \frac{r e^{-r^2}}{\frac12  - \frac{\log r}{\log \eps} }\,dr<\infty.
\end{equation}
Note that, given the claim, from \eqref{almost last bound} we achieve 
$$
 \|\omega^\nu (t)\|^2_{L^2} \lesssim \frac{1}{\nu t |\log (\nu t)|} \qquad \text{if } \nu t<1,
$$
from which, for any $\delta\in (0,1)$, we deduce\footnote{Recall that $\log s\sim s -1 $ as $s\rightarrow 1$.}
\begin{align}
\nu \int_{\delta}^1 \|\omega^{\nu}(\tau)\|_{L^2}^2\,d\tau &\lesssim - \log \bigg(\frac{\log \frac{1}{\nu}}{\log \frac{1}{\nu}+ \log \frac{1}{\delta}}\bigg) \\
&\sim \frac{\log \frac{1}{\delta}}{\log \frac{1}{\nu}+\log \frac{1}{\delta}} \\
&\leq \frac{|\log \delta |}{|\log \nu|} 
\end{align}
if $\nu$ is sufficiently small, depending also on $\delta$. It remains to prove the claim.  Let us define 
$$
f_{\eps}(r):= \frac{re^{-r^2}}{\frac{1}{2}-\frac{\log r}{\log \eps}}\mathbbm{1}_{[4 \sqrt{\eps}, \infty)}(r)\qquad \forall r>0.
$$
We treat separately the case $r>4$ and $r\in [4\sqrt{\eps}, 4]$. For $r>4$ we have $f_\eps (r)\leq 2 r e^{-r^2}$, from which 
$$
\sup_{\eps\in (0,1)}\int_4^\infty f_\eps (r)\,dr \leq 2\int_4^\infty r e^{-r^2}\,dr<\infty.
$$
For $r\in [4\sqrt{\eps},4]$ we instead use that 
$$
\int_{4\sqrt{\eps}}^4 f_\eps(r)\,dr \leq \int_{4\sqrt{\eps}}^4 \frac{r}{\frac{1}{2}-\frac{\log r}{\log \eps}}\,dr.
$$
The derivative of the function inside the latter integral is given by 
$$
\frac{\log \eps - 2\log r +2 }{2\log \eps \left(\frac{1}{2}-\frac{\log r}{\log \eps}\right)^2},
$$
which is positive if and only if $\log \eps - 2\log r +2<0$, that is, if and only if $r>e \sqrt{\eps}$. It follows that the function is increasing in the interval $[4\sqrt{\eps},4]$, with the maximum achieved at the endpoint $r=4$. Therefore
$$
\sup_{\eps \in (0,1)}\int_{4\sqrt{\eps}}^4 f_\eps(r)\,dr \leq 16.
$$
This concludes the proof of the claim \eqref{claim integral bounded}.

To summarize, even if the initial datum \eqref{def of L1 vort} has a fixed total mass and saturates the logarithmic decay on balls, the corresponding solutions to \eqref{NS-Vort} enjoy, for any $\delta\in (0,1)$, the estimate
$$
\nu \int_{\delta}^1 \|\omega^{\nu}(\tau)\|_{L^2}^2\,d\tau\lesssim \frac{1}{|\log \nu|}\qquad \forall 0<\nu\ll 1,
$$
failing again to satisfy \eqref{attempt dissipation sharp} by a square. 

\subsection{Conclusions}\label{S:conclusions} We have tried to achieve \eqref{attempt dissipation sharp} with two examples. The first, given in \cref{S:first ex}, was obtained by appropriately rescaling a smooth and compactly supported vorticity profile, so that it concentrates at the origin. In order for the vorticity to have the desired logarithmic decay on balls \eqref{attempt M log decay}, and also for the velocity to have finite kinetic energy, the specific choice \eqref{rescaling heat} for the rescaling is necessary.  However, this makes the total mass of the initial data vanishing like $\lesssim |\log \nu|^{-\sfrac12}$, preventing the possibility to achieve \eqref{attempt dissipation sharp} because of the general estimate \eqref{general bound with L1 initial mass}.

This motivated us to study, in \cref{S:second ex}, the case with a fixed initial datum. We have thus tried with the integrable function \eqref{def of L1 vort}, since it has the sharp decay on balls. However, even in this case, the dissipation for positive times turned out to be $\lesssim |\log \nu|^{-1}$.
This leads us to believe that, in the case it would really be possible to achieve \eqref{attempt dissipation sharp},  giving therefore a positive answer to Conjecture \ref{Conj}, the vorticity should concentrate on a ``sparser'' set and not only to a single point, nor on a too ``organized'' one. This is indeed how the example from Proposition  \ref{P:nash ineq sharp} saturating the Nash inequality looks like. If this intuition is correct, it must be then necessary to give up on the possibility to saturate the rate of the dissipation for \eqref{NS-Vort} with the heat equation because of the non-local interaction between the different pieces of the concentration set of the vorticity that would make the nonlinear term $u^\nu\cdot \nabla \omega^\nu$ nontrivial\footnote{Note that, as shown in Proposition \ref{P:saturate algebraic rate}, this is not the case for measure vorticities with the linear decay \eqref{1d decay on balls} on balls, where the rate can be saturated with radial solutions to the heat equation.}. Having a reasonable control over the nonlinear dynamics of such singular vorticity measures seems a quite challenging problem.

\bibliographystyle{plain} 
\bibliography{biblio}

\end{document}